\documentclass[12pt]{amsart}
\usepackage[T1]{fontenc}
\usepackage{lmodern,amsthm,amsmath,amssymb,mathrsfs,mathtools,booktabs,longtable,array}
\mathtoolsset{showonlyrefs=true}
\usepackage[a4paper,left=2cm,right=2cm,top=2cm,bottom=2cm]{geometry}

\usepackage{cite}
\usepackage[unicode,pdfencoding=auto,colorlinks=true,linkcolor=blue,citecolor=blue,urlcolor=blue]{hyperref}

\allowdisplaybreaks[2]
\makeatletter
\@ifundefined{subjclassname@2020}{\@namedef{subjclassname@2020}{\textup{2020} Mathematics Subject Classification}}{}
\makeatother
\newcommand{\F}{\mathbb F_2}
\newcommand{\Pk}{\mathcal P_k}

\newcommand{\A}{\mathcal A}
\newcommand{\Aplus}{\mathcal A^+}
\newcommand{\Qk}{Q\mathcal P_k}
\newcommand{\Qkd}{(Q\mathcal P_k)_d}
\newcommand{\Sq}{\operatorname{Sq}}
\newcommand{\LM}{\operatorname{LM}}
\newcommand{\Supp}{\operatorname{Supp}}
\newcommand{\vsupp}{\operatorname{vsupp}}
\newcommand{\Rel}{\operatorname{Rel}}
\newcommand{\rank}{\operatorname{rank}}
\newcommand{\GL}{\operatorname{GL}}
\newcommand{\DP}{\operatorname{DP}}
\newcommand{\im}{\operatorname{im}}
\newcommand{\Span}{\operatorname{span}_{\F}}
\newcommand{\Z}{\mathbb Z}
\newcommand{\NN}{\mathbb N}
\newcommand{\lamstd}{\lambda_{\mathrm{std}}}

\newcommand{\eps}{\varepsilon}
\newtheorem{theorem}{Theorem}[section]
\newtheorem{corollary}[theorem]{Corollary}
\newtheorem{lemma}[theorem]{Lemma}
\newtheorem{proposition}[theorem]{Proposition}
\theoremstyle{definition}
\newtheorem{definition}[theorem]{Definition}
\theoremstyle{remark}
\newtheorem{remark}[theorem]{Remark}
\newtheorem{example}[theorem]{Example}
\def\DD{D\kern-.7em\raise0.4ex\hbox{\char '55}\kern.33em}
\newcommand{\DVPhuc}{\DD. V. Ph\'uc}

\title[Two-sided bounds for cohit dimensions]{Two-sided dimension bounds\\ for the Peterson hit problem\\ via projections and matrix minors}

\author[\DD. V. Ph\'uc]{\DD\d{\u a}ng V\~o Ph\'uc$^{1,*}$}
\address{$^{1}$Department of Mathematics, FPT University, An Phu Thinh New Urban Area, Quy Nhon, Vietnam}
\address{$^{1}$Department of Mathematics and Statistics, Quy Nhon University, 170 An Duong Vuong, Quy Nhon, Vietnam}
\email{dangphuc150488@gmail.com, dangvophuc@qnu.edu.vn}
\thanks{$^{*}$Corresponding author}
\thanks{ORCID (\DD. V. Ph\'uc): \url{https://orcid.org/0000-0002-6885-3996}}

\author[P. V. Danchev]{Peter V. Danchev$^{2}$}
\address{$^{2}$Institute of Mathematics and Informatics, Bulgarian Academy of Sciences, Bulgaria}
\email{danchev@math.bas.bg; pvdanchev@yahoo.com}
\thanks{ORCID (P. V. Danchev): \url{https://orcid.org/0000-0002-2016-2336}}

\date{November 26, 2026}

\begin{document}

\begin{abstract}
Let $\mathcal P_k=\mathbb F_2[x_1,\ldots,x_k]$, with $\deg x_i=1$, be the polynomial algebra equipped with the standard unstable action of the mod-$2$ Steenrod algebra $\A$. In algebraic topology, the \textit{Peterson hit problem} asks for a minimal homogeneous set of $\A$-generators for $\mathcal P_k$, or equivalently for a homogeneous basis of the cohit space $Q\mathcal P_k=\mathcal P_k/\Aplus\mathcal P_k = \mathbb F_2\otimes_{\A}\mathcal P_k$, where $\Aplus$ denotes the augmentation ideal of $\A$. The dimension of $(Q\mathcal P_k)_d$ therefore gives the number of generators required in positive degree $d$. Although solved in every degree $d$ for $k\leq 4$, the problem remains open in general. As the number of variables $k$ and the degree $d$ increase, exact calculations involve increasingly large systems of Steenrod relations. As a partial quantitative contribution to this open problem, we establish explicit upper and lower bounds for $\dim_{\mathbb F_2}(Q\mathcal P_k)_d$ for arbitrary positive integers $k$ and $d$, determining the exact dimension whenever these bounds coincide. Our method combines binary combinatorics and linear algebra with graph-theoretic and simplicial structures in the matrix of the generating Steenrod squares. The results have also been independently and formally verified in Lean 4.
\end{abstract}

\maketitle

\tableofcontents

\section{Introduction}\label{sec:intro}

The mod-$2$ cohomology of an elementary abelian $2$-group provides a natural setting for the study of generators over the mod-$2$ Steenrod algebra $\A$. Specifically, let $V_k = (\Z/2)^k$ denote the elementary abelian $2$-group of rank $k \geq 1$. The cohomology algebra of its classifying space $BV_k$ is the polynomial ring
\[
 \Pk = H^*(BV_k;\F) \cong \F[x_1,\ldots,x_k],\qquad \deg x_i=1.
\]
The total Steenrod square satisfies $\Sq(x_i)=x_i+x_i^2$, and the Cartan formula determines its action on every polynomial. One of the most intensely studied and challenging open problems in this area is the classical \textit{Peterson hit problem}, which seeks a minimal homogeneous generating set for $\Pk$ viewed as an unstable module over $\A$. Equivalently, it asks for a basis in each degree of the cohit space
\[
 \Qk = \F \otimes_{\A} \Pk \cong \Pk / \Aplus \Pk,
\]
where $\Aplus$ is the augmentation ideal and $\Aplus\Pk$ is the subspace of hit polynomials. The space $\Qk$ inherits a natural grading from $\Pk$. Letting $\Qkd$ denote its degree-$d$ homogeneous component, the dimension $\dim_{\F}\Qkd$ is exactly the number of degree-$d$ generators in a minimal homogeneous generating set. The two monographs of Walker and Wood \cite{WalkerWoodI,WalkerWoodII} develop the hit problem through the polynomial action, binary weights, and representations of finite general linear groups.

The numerical problem is part of a broader representation-theoretic question. Linear substitutions commute with Steenrod operations, so each cohit space is a $\GL(k,\F)$-module. Singer's algebraic transfer relates the coinvariants of the dual Steenrod kernel to the cohomology of $\A$ \cite{Singer1989}; the Dickson-algebra hit problem gives a related question for invariant polynomials \cite{HungNam2001}. The same polynomial action occurs in the study of unstable modules and algebras modulo nilpotent objects \cite{HennLannesSchwartz1993}. Estimating the dimension of the full cohit space is therefore a first step towards understanding the representations in which these constructions take place.

Previous calculations completely determine the cohit spaces $\Qk$ in every degree for $k\leq4$; see \cite{WalkerWoodI,WalkerWoodII} and the references cited therein. The distinction between small and larger numbers of variables is already visible in the problem of finding sharp dimension bounds. For fixed $k$, the dimensions $\dim\Qkd$ are bounded independently of $d$ \cite[Theorem~7.1.1]{WalkerWoodI}. Kameko conjectured the sharper bound
\[
 \dim_{\F}\Qkd\leq\prod_{i=1}^{k}(2^i-1)
 \qquad\text{for all }k\geq1\text{ and }d\geq0,
\]
whose right-hand side is the number of complete flags in $\F^k$. Although this bound holds for $k\leq4$, counterexamples have been constructed for every $k\geq5$, with a partial-flag interpretation provided in \cite[\S\S24.5--24.6 and references therein]{WalkerWoodII}. The failure of this proposed bound does not contradict uniform boundedness. However, it demonstrates that for higher ranks, the cohit dimensions cannot be controlled by the complete-flag count alone. Consequently, establishing effective bounds for individual degrees remains an independent and challenging aspect of the hit problem.

Several structural reductions precede the explicit computation of $\dim_{\F}\Qkd$. If $\alpha(n)$ denotes the number of ones in the binary expansion of $n$, put
\[
 \mu(n)=\min\{r\in\mathbb Z_{\geq0}:\alpha(n+r)\leq r\},\qquad n\geq0.
\]
Wood's vanishing theorem establishes that $\Qkd=0$ when $\mu(d)>k$. When $\mu(d)=k$, Kameko's isomorphism reduces degree $d$ to $(d-k)/2$. Iteration reduces the remaining computation of $\dim_{\F}\Qkd$ to the core degrees satisfying $\mu(d)<k$; see \cite[Theorem~2.5.5 and \S6.5]{WalkerWoodI}. Even within this restricted range, the relations among monomials remain highly intricate. To systematically organize these relations, researchers employ algebraic tools such as weight filtrations, minimal spikes, and restricted hit equations. Leveraging these specific structures, the spike and dimension estimates of Mothebe \cite{Mothebe2008Spikes,Mothebe2013Dimension}, alongside the monomial constructions and low-degree formulas in \cite{Mothebe2013Admissible,MothebeUys2015,Mothebe2016Products,MothebeKMR2016}, provide explicit information without requiring a uniform treatment of every degree.

While these structural methods simplify specific cases, determining the exact cohit dimension in an arbitrary degree ultimately relies on a direct linear-algebraic computation. Specifically, this dimension is governed by the rank of the degree-$d$ hit matrix $M(k,d)$. The columns of $M(k,d)$, indexed by the pairs $(t,g)$, represent the monomial expansions of $\Sq^{2^t}(g)$ with $\deg g=d-2^t$. Writing $N_k(d)=\binom{d+k-1}{k-1}$, one has
\[
 \dim_{\F}\Qkd=N_k(d)-\rank M(k,d); \ \ \mbox{(see Remark \ref{rein} in Section \ref{sec:matrix})}.
\]
Although the ambient dimension $N_k(d)$ grows polynomially in $d$ for a fixed $k$, the large number of columns and the intricate linear dependencies make a complete rank calculation computationally expensive. To circumvent this, a dimension inequality can be derived from more limited data: a nonsingular minor provides an upper cohit bound, whereas linear relations among the columns or dual functionals annihilating them establish a lower bound. Building on this principle, the present work establishes explicit two-sided bounds for $\dim_{\F}\Qkd$, for $k\geq1$ and $d\geq0$. By utilizing selected minors of $M(k,d)$ and dual functionals annihilating its image, we bypass the need to first determine a complete admissible basis. A brief summary of our main results is as follows. The notation and constructions used below are defined in Section~\ref{sec:matrix}.

\begin{itemize}

\item[(1)] The zero rows of $M(k,d)$ are precisely those indexed by degree-$d$ spike monomials (Theorem~\ref{thm:exact-zero-rows} and Proposition~\ref{prop:spike-count} in Section~\ref{sec:upper}). The zero columns of each $\Sq^{2^t}$ layer admit an exact count determined by the residues modulo $2^{t+1}$ of the exponents of the degree-$(d-2^t)$ monomials (Proposition~\ref{prop:Z} in Section~\ref{sec:upper}). These counts give upper bounds for the hit rank. Further bounds use the exact rank of $\Sq^1$ and evaluated Adem relations, with overlaps accounted for by computing the dimension of the sum of the relevant relation spaces (Proposition~\ref{prop:contraction} and Theorem~\ref{thm:Adem-upper} in Section~\ref{sec:upper}).

\medskip

\item[(2)] Complete closed formulas count all labelled singleton columns and their distinct output monomials (Theorem~\ref{thm:all-singletons} in Section~\ref{sec:lower}). Triangular minors arise from acyclic choices of pivots in the column supports, and the largest size obtainable by varying the total row order equals the maximum cardinality of a uniquely restricted matching in the bipartite support graph (Theorem~\ref{thm:optimal-order} in Section~\ref{sec:lower}). The selected singleton union contributes to the leading-row bound for every total row order (Proposition~\ref{prop:nested-singletons}), while column families whose incidence graphs are forests and whose columns each have at least two nonzero entries are linearly independent and admit acyclic pivots (Proposition~\ref{lem:Berge-independent}). Certain congruence families of columns decompose into components whose matrices are augmented simplicial boundary matrices with explicit rank formulas (Proposition~\ref{prop:q-components} and Theorem~\ref{thm:simplicial-rank} in Section~\ref{sec:lower}), and the maximum of their summed ranks over the layers equals the first-square rank already detected by the lexicographic leading rows in the first layer (Proposition~\ref{prop:simplicial-collapse}).

\medskip

\item[(3)] For $d>0$, an explicit projection with kernel $\im(\Sq^1:\mathcal P_k^{d-1}\to\mathcal P_k^d)$ gives a matrix $\overline M(k,d)$ satisfying
\[
 \rank M(k,d)=\rho_k(d-1)+\rank\overline M(k,d),
\]
where $\rho_k(n)$ is the rank of $\Sq^1:\mathcal P_k^n\to\mathcal P_k^{n+1}$ (Theorem~\ref{thm:Sq1-quotient} in Section~\ref{sec:projection}). We count the columns that vanish after projection (Proposition~\ref{prop:projected-zero-columns}) and classify the projected zero rows, obtaining a closed formula for their number $z_k(d)$ and pulling back their coordinate functionals to an independent family in the classical dual-spike module $J_k(d)$ (Proposition~\ref{prop:projected-dual}, Theorem~\ref{thm:projected-zero-classification}, and Corollary~\ref{cor:projected-spike-module}). Combining these dual functionals over all variable orders yields a subspace $\mathcal L_k^{\mathrm{ord}}(d)\subseteq J_k(d)$ with an explicit dimension formula (Proposition~\ref{prop:all-variable-orders} in Section~\ref{sec:projection}), and products on disjoint variable pairs indexed by standard two-row tableaux extend this construction to an independent family of cardinality $\beta_k(d)$ in $J_k(d)$ (Theorem~\ref{thm:disjoint-pair-bound}). Altogether, for $k\geq1$ and $d>0$, these constructions yield the chain of explicitly evaluable lower bounds
\[
 \mathcal S_k(d)\leq z_k(d)\leq
 \dim\mathcal L_k^{\mathrm{ord}}(d)\leq
 \beta_k(d)\leq\dim J_k(d)\leq\dim(Q\mathcal P_k)_d,
\]
where $\mathcal S_k(d)$ counts degree-$d$ spikes (Sections~\ref{subsec:prelim-spikes}--\ref{subsec:prelim-tableaux}). Here the coordinate spikes together with the one-pair terms account for $\dim\mathcal L_k^{\mathrm{ord}}(d)$, while terms built from at least two disjoint variable pairs can strictly enlarge the constructed independent family.

\medskip

\item[(4)] Combining the unprojected and projected minors, relation spaces, and annihilating dual functionals gives two-sided rank and dimension bounds for all $k\geq1$ and $d\geq0$, with $\dim(Q\mathcal P_k)_0=1$ (Theorem~\ref{thm:main-rank} in Section~\ref{sec:main}). The degree-six evaluations in Section~\ref{subsec:degree-six} exhibit strict improvements along the dual chain and cases where the resulting lower bound coincides with a projected minor upper bound, while Section~\ref{subsec:k5d29} treats $(k,d)=(5,29)$ and shows how projection sharpens the upper bound and admissible-monomial insertion supplies further lower-bound information.
\end{itemize}

\medskip
\noindent\textbf{Methodology and Techniques.}
The bounds combine linear algebra over $\F$ with binary combinatorics, graph theory, and simplicial homology, avoiding the construction of a complete admissible basis. The Cartan formula and Lucas' criterion determine the column supports explicitly (Sections~\ref{subsec:prelim-orders}--\ref{subsec:prelim-graphs}). Rather than fixing a single monomial order, we characterize acyclic pivot assignments across all row orders via uniquely restricted matchings in the bipartite support graph, which are defined by the absence of cycles alternating between matching and nonmatching edges \cite{GolumbicHirstLewenstein2001}. For congruence families of columns, the Frobenius identity $\Sq^{2^t}(f^{2^t})=(\Sq^1f)^{2^t}$ links their augmented simplicial boundary components directly to first-square ranks (Section~\ref{subsec:prelim-simplicial}).

To remove the first-square image before inspecting higher layers, we use the positive-degree exactness of $(\Pk,\Sq^1)$ \cite[Proposition~1.3.5]{WalkerWoodI} to specify an explicit contracting homotopy $h$ and projection $p_d=h\Sq^1$ onto $T_d(k)$, the span of degree-$d$ monomials whose first positive exponent in the fixed variable order is odd (Sections~\ref{subsec:prelim-contraction}--\ref{subsec:prelim-projection}). For $t\geq1$, the Adem identity $\Sq^1\Sq^{2^t}=\Sq^{2^t+1}$ gives $p_d\Sq^{2^t}(g)=h\Sq^{2^t+1}(g)$, allowing projected columns to be formed directly from the polynomial source spaces without prior Gaussian elimination. In the context of finite-subalgebra hit problems such as \cite{Ault2014}, the starting quotient $T_d(k)$ is the degree-$d$ $\A(0)$-cohit space, and the projected layers with $1\leq t\leq n$ present the $\A(n)$-cohit space in these coordinates, culminating in the full presentation $(Q\mathcal P_k)_d\cong T_d(k)/\im\overline M(k,d)$.

Although $T_d(k)$ depends on the variable ordering and is not a $\GL(k,\F)$-submodule, this non-equivariance becomes an advantage on the dual side. Pulling back projected zero-row functionals and varying the variable ordering generates $\GL(k,\F)$-translates inside the classical dual-spike module $J_k(d)$ \cite{WalkerWoodI,WalkerWoodII}. Multiplying degree-three dual-spike combinations across disjoint variable pairs and expanding the divided-power products expresses these elements as sums of $\GL(k,\F)$-translates of coordinate dual spikes, while standard two-row tableaux select a triangular subfamily whose linear independence follows directly from a leading-monomial argument (Sections~\ref{subsec:prelim-dual}--\ref{subsec:prelim-tableaux}).

Finally, the resulting estimates apply for every $k\geq1$ and $d>0$ without prior knowledge of a weight quotient or Kameko kernel, while remaining compatible with classical reductions and insertion results from \cite{WalkerWoodI,WalkerWoodII} when sharper bounds are needed. Remark~\ref{rem:milnor-comparison} in Section~\ref{sec:matrix} also distinguishes the polynomial action matrices studied here from the integer indexing matrices in Milnor's multiplication formula.

\medskip

The paper is organized as follows. Section~\ref{sec:matrix} fixes the hit matrix and the support and order conventions. Section~\ref{sec:upper} derives upper rank bounds from zero rows, zero columns, and relations among the squares. Section~\ref{sec:lower} constructs and compares the minor bounds. Section~\ref{sec:projection} gives the explicit quotient by the first square, classifies its zero rows, and identifies the associated dual functionals. Finally, Section~\ref{sec:main} combines these results into two-sided cohit inequalities, evaluates them in specified degrees, and compares the resulting bounds with known dimensions.

\section{Preliminaries}\label{sec:matrix}

To help the reader follow our approach to the results in the subsequent sections as clearly as possible, we present here all the notation, conventions, constructions, and counting functions needed in our arguments, together with explicit examples. Our primary references for the polynomial action of the Steenrod algebra, weight orders, spikes, divided-power duality, and dual-spike modules are the two-volume monograph of Walker and Wood \cite{WalkerWoodI,WalkerWoodII}. Additional background terminology for Milnor multiplication matrices, uniquely restricted matchings, augmented simplicial boundary maps, two-row tableaux, and admissible-monomial insertion follows \cite{Milnor1958,GolumbicHirstLewenstein2001,Hatcher2002,James1978,MothebeUys2015}, respectively. While the underlying algebraic and combinatorial structures are standard, the labelled hit matrix, the explicit contracting homotopy, and the first-square projection provide coordinates for the calculations in this paper. Each group of definitions indicates where it will be used, and worked computations in two or three variables illustrate the distinction between a column label, its monomial outputs, and the linear relations among different columns.

\subsection{Polynomial degrees, binary digits, and Steenrod squares}\label{subsec:prelim-basic}

Throughout, $\F=\mathbb F_2$, $k\geq1$, $d\geq0$, and $\NN=\{0,1,2,\ldots\}$. All vector spaces, matrix ranks, and tensor products are over $\F$ unless stated otherwise. Counting functions take integer values, and their generating functions are interpreted in $\Z[[\zeta]]$. Write $[k]=\{1,\ldots,k\}$ and $\mathcal P_k=\F[x_1,\ldots,x_k]$, with each variable of degree one. Its degree-$n$ part is $\mathcal P_k^n$, which is zero for $n<0$, and
\[
 N_k(n)=\dim\mathcal P_k^n=
 \begin{cases}\binom{n+k-1}{k-1},&n\geq0,\\0,&n<0.\end{cases}
\]
Indeed, degree-$n$ monomials correspond to weak compositions, that is, ordered tuples of nonnegative integers with sum $n$. A binomial coefficient with nonnegative upper entry is zero when its lower entry is negative or larger than its upper entry. Empty sums and maxima of empty families of nonnegative bounds are zero; empty products are one. These conventions apply in Sections~\ref{sec:upper}--\ref{sec:main}, including all boundary cases.

For an exponent tuple $e=(e_1,\ldots,e_k)\in\NN^k$, put $x^e=\prod_{i=1}^k x_i^{e_i}$ and $|e|=\sum_i e_i$. Its variable support is $\vsupp(e)=\{i:e_i>0\}$. For $I\subseteq[k]$, let $\mathbf1_I$ be its indicator tuple, and put $\eps_i=\mathbf1_{\{i\}}$. When a tuple is nonzero, $i(e)=\min\vsupp(e)$ is its first positive position. These symbols are used in Sections~\ref{sec:upper}--\ref{sec:main}. For example, $e=(0,2,1)$ represents $x_2^2x_3$, has degree three and variable support $\{2,3\}$, and satisfies $i(e)=2$ and $e-\eps_2=(0,1,1)$.

Write $n=\sum_{j\geq0}\alpha_j(n)2^j$, where $\alpha_j(n)\in\{0,1\}$, and let $\alpha(n)=\sum_j\alpha_j(n)$ be the number of binary ones. For nonnegative integers $n',n$, write $n'\preceq n$ when $\alpha_j(n')\leq\alpha_j(n)$ for every $j$. Thus $2\preceq3$, whereas $2\npreceq5$, since $3=(11)_2$ and $5=(101)_2$. The parity criterion
\[
 \binom{n}{n'}\equiv1\pmod2\quad\Longleftrightarrow\quad n'\preceq n
\]
is the binary binomial criterion in \cite[\S1.4, Proposition~1.4.11, p.~11]{WalkerWoodI}. The binary relation is used in Sections~\ref{sec:upper}--\ref{sec:main}; the function $\alpha$ also occurs in the insertion hypothesis in Section~\ref{subsec:k5d29}.

The standard left action of the Steenrod squares $\Sq^\vartheta$ ($\vartheta\geq0$) on $\mathcal P_k$ is uniquely determined by $\F$-linearity, the Cartan formula $\Sq^\vartheta(fg)=\sum_{u+v=\vartheta}\Sq^u(f)\Sq^v(g)$ for $f,g\in\mathcal P_k$, and the values on the degree-one generators
\[
 \Sq^0(x_i)=x_i,\qquad \Sq^1(x_i)=x_i^2,\qquad \Sq^\vartheta(x_i)=0\quad (\vartheta>1)
\]
for $1\leq i\leq k$. By induction on $a_i$, this gives $\Sq^{b_i}(x_i^{a_i})=\binom{a_i}{b_i}x_i^{a_i+b_i}$ over $\F$. Combining this identity with the Cartan formula and the binary binomial criterion above, for any exponent tuple $a=(a_1,\ldots,a_k)\in\NN^k$ and any integer $\vartheta\geq0$, we obtain
\begin{equation}\label{eq:cartan}
 \Sq^\vartheta(x^a)=\sum_{\substack{b\in\NN^k,\ |b|=\vartheta\\b_i\preceq a_i\ (1\leq i\leq k)}}x^{a+b}.
\end{equation}
Here $b=(b_1,\ldots,b_k)$ represents an allocation of the total square degree $\vartheta$ among the variables, and the condition $b_i\preceq a_i$ characterizes precisely the terms whose binomial coefficients $\binom{a_i}{b_i}$ are odd. We use the phrase Cartan--Lucas allocation for such a tuple, solely as a description of the conditions in \eqref{eq:cartan}. Distinct allocations have distinct output exponent tuples $a+b$, so no cancellation occurs within the square of a single monomial. The underlying formulas and instability, $\Sq^\vartheta(\mathcal P_k^n)=0$ for $\vartheta>n$, are \cite[\S1.1, Propositions~1.1.4 and 1.1.8--1.1.11]{WalkerWoodI}. This expansion is used in every later section.

For example, take $a=(1,3)$ and $\vartheta=2$. The only permitted allocations are $b=(0,2)$ and $(1,1)$: the alternative $(2,0)$ fails because $2\npreceq1$. Consequently
\[
 \Sq^2(x_1x_2^3)=x_1x_2^5+x_1^2x_2^4.
\]
Both terms have coefficient one, even though their total coefficient sum is zero in $\F$. Thus coefficient-sum parity does not test whether a column is zero.

Products of operations act from the left, so $(\Sq^\vartheta\Sq^{\vartheta\prime})(f)=\Sq^\vartheta(\Sq^{\vartheta\prime}(f))$. The elements $\Sq^{2^t}$, $t\geq0$, generate $\A$ by \cite[\S3.2, Proposition~3.2.4]{WalkerWoodI}. In particular,
\begin{equation}\label{eq:outer-generators}
 \Aplus=\sum_{t\geq0}\Sq^{2^t}\A.
\end{equation}
The summands are right ideals, and the displayed generator is the outermost operation. A polynomial is hit when it belongs to $\Aplus\mathcal P_k$, and $\Qkd=\mathcal P_k^d/(\Aplus\mathcal P_k)_d$ is its cohit quotient; these agree with \cite[\S1.4, Definitions~1.4.1 and 1.4.5]{WalkerWoodI}. They are used in Sections~\ref{sec:upper}--\ref{sec:main}. For instance, $x_1^2=\Sq^1(x_1)$ is hit, whereas no positive square produces $x_1$.

\subsection{The labelled hit matrix and its monomial outputs}\label{subsec:prelim-matrix}

\begin{definition}\label{def:matrix}
Let $\mathcal R_d(k)=\{x^e:|e|=d\}$ and $V(k,d)=\bigoplus_{2^t\leq d}\mathcal P_k^{d-2^t}$. The hit matrix $M(k,d)$ is the matrix, in monomial bases, of
\begin{equation}\label{eq:Phi}
 \Phi_{k,d}:V(k,d)\longrightarrow\mathcal P_k^d,
 \qquad (g_t)_t\longmapsto\sum_t\Sq^{2^t}(g_t).
\end{equation}
A column has the label $(t,x^a)$, where $|a|=d-2^t$. The $t$th layer is the family of columns with this fixed value of $t$, equivalently the matrix of $\Sq^{2^t}:\mathcal P_k^{d-2^t}\to\mathcal P_k^d$. Labels remain distinct even when their coordinate vectors coincide. The support $\Supp(c)$ of a column $c$ is the set of monomials whose coefficients in that column are one. Each such monomial is an output monomial, also called a target of that column. A singleton column has support of cardinality one, and its unique output is its singleton target. These definitions are used in Sections~\ref{sec:upper}--\ref{sec:main}.
\end{definition}

The entry at row $x^e$ and column $(t,x^a)$ is
\[
 M(k,d)_{e,(t,a)}=
 \begin{cases}
 \displaystyle\prod_{i=1}^k\binom{a_i}{e_i-a_i}\pmod2,&e_i\geq a_i\text{ for every }i,\\
 0,&\text{otherwise},
 \end{cases}
 \qquad |e|=d,\quad |a|=d-2^t.
\]
Thus a target is an individual monomial in a column, rather than the whole codomain. The total number of labelled columns is $c_k^{\mathrm{tot}}(d)=\dim V(k,d)=\sum_{2^t\leq d}N_k(d-2^t)$, which will be used in Section~\ref{sec:upper}. Layers with $2^{t+1}>d$ consist of zero columns by instability, but their labels are retained.

For $M(2,3)$, order the rows as $(x_1^3,x_1^2x_2,x_1x_2^2,x_2^3)$. The column labels, grouped by layer, are
\[
 \underbrace{(0,x_1^2),\ (0,x_1x_2),\ (0,x_2^2)}_{t=0},\qquad
 \underbrace{(1,x_1),\ (1,x_2)}_{t=1}.
\]
The resulting matrix is
\[
 M(2,3)=
 \begin{pmatrix}0&0&0&0&0\\0&1&0&0&0\\0&1&0&0&0\\0&0&0&0&0\end{pmatrix}.
\]
The only nonzero column has two targets. In contrast, $\Sq^1(x_1^3)=x_1^4$ and $\Sq^2(x_1^2)=x_1^4$ are two differently labelled singleton columns of $M(2,4)$ with the same singleton target.

\begin{remark}\label{rein}
The image of $\Phi_{k,d}$ is $(\Aplus\mathcal P_k)_d$, and therefore
\begin{equation}\label{eq:dimension-rank}
 \dim\Qkd=N_k(d)-\rank M(k,d).
\end{equation}
Every column is a positive Steenrod image, which proves one inclusion. Conversely, expand each positive homogeneous operation in a hit expression into monomials in the generators $\Sq^{2^t}$. For a monomial of operations whose outermost factor is $\Sq^{2^t}$, its remaining factors applied to the polynomial produce an element of degree $d-2^t$. Expanding that element in monomials expresses the original summand as a linear combination of columns in layer $t$. This proves the reverse inclusion. Finally, a quotient of a finite-dimensional space has dimension equal to the dimension of that space minus the dimension of the subspace being divided out, giving \eqref{eq:dimension-rank}. This identity is used throughout Sections~\ref{sec:upper}--\ref{sec:main}.
\end{remark}

\begin{remark}[Polynomial action matrices and Milnor matrices]\label{rem:milnor-comparison}

The matrix $M(k,d)$ is formed by concatenating, with respect to the monomial bases, the matrices of the maps
\[
 \Sq^{2^t}\colon (\mathcal P_k)_{d-2^t}\longrightarrow(\mathcal P_k)_d,
 \qquad t\geq0,\quad 2^t\leq d.
\]
Thus its columns record the coefficients of the images of the corresponding monomials under the generating Steenrod squares. The Milnor matrices appearing in \cite[\S6, Theorem~4b, and Appendix~1]{Milnor1958} and \cite[\S4.1, Definition~4.1.1 and Theorem~4.1.2]{WalkerWoodI} are arrays of nonnegative integers with only finitely many nonzero entries, used to index the terms in the product of two Milnor basis elements of $\A$. Their weighted row sums and ordinary column sums specify the two input basis indices, their diagonal sums specify the output basis index, and products of multinomial coefficients determine each array's contribution modulo $2$. This distinguishes the matrices over $\mathbb F_2$ representing polynomial actions from the integer arrays indexing multiplication in $\A$. The rank arguments in Sections~\ref{sec:upper}--\ref{sec:main} concern $M(k,d)$, its projected form $\overline M(k,d)$, and the associated minors, relation spaces, and dual functionals.

The rank--nullity theorem, spike obstruction, positive-degree exactness of $\Sq^1$, rank bound from uniquely restricted matchings, and simplicial boundary rank formulas are classical inputs. The subsequent arguments use these inputs to derive complete singleton column and target counts, give explicit coordinates for the quotient by $\im\Sq^1$, classify and count the zero rows and columns of the projected matrix, and construct explicit minors and independent families of annihilating dual functionals. These dual families lie in the classical dual-spike module $J_k(d)$ defined in Section~\ref{subsec:prelim-dual}; their construction and proved independence yield explicitly evaluable lower bounds for its dimension.

\end{remark}

\subsection{Weight orders and leading rows}\label{subsec:prelim-orders}

The weight sequence records the number of ones at each binary position among the exponents. For $x^e$, it is
\[
 \omega(e)=(\omega_1(e),\omega_2(e),\ldots),\qquad \omega_j(e)=\sum_i\alpha_{j-1}(e_i),
\]
with trailing zeros understood. This is the $\omega$-sequence of \cite[\S5.1, Definition~5.1.5]{WalkerWoodI}. The exponent orders use $x_1>\cdots>x_k$. In fixed degree, exponent lexicographic order compares the first unequal exponent in the usual direction. Degree-reverse-lexicographic order compares the last unequal exponent in the opposite direction. The left weight order compares the first unequal weight entry in the usual direction; the right weight order compares the last unequal weight entry in the opposite direction. Both weight orders use exponent lexicographic order to resolve equal weight sequences. In particular, if $j$ is the largest position where $\omega_j\ne\omega_j^{\prime}$, then $\omega<_{r}\omega^{\prime}$ means $\omega_j>\omega_j^{\prime}$. The sequence-order conventions are exactly those of \cite[\S3.1, Definition~3.1.5]{WalkerWoodI}. These four orders are used in Sections~\ref{sec:lower}--\ref{sec:main}.

For example, $x_1^2x_2^3$ and $x_1x_2^4$ have weights $(1,2)$ and $(1,0,1)$, respectively. The former weight is larger in both weight orders: its second entry is larger in the left comparison, and its third entry is smaller in the reversed right comparison. In three variables, $(1,0,2)$ is larger than $(0,3,0)$ in exponent lexicographic order and smaller in degree-reverse-lexicographic order. Thus the two exponent orders can differ even in a small degree.

For a finite matrix $\mathsf A$ with a total order $\prec$ on its rows, let $\LM_\prec(c)$ denote the largest support row of a nonzero column $c$, and let $\lambda_\prec(\mathsf A)$ count its distinct leading rows. When the rows are monomials of one fixed degree, or a subset of those monomials, the four prescribed orders also define $\lamstd(\mathsf A)$ as follows:
\begin{equation}\label{eq:lambda-order}
 \lambda_\prec(\mathsf A)=\#\{\LM_\prec(c):c\ne0\},\qquad
 \lamstd(\mathsf A)=\max_{\prec\in\mathcal O}\lambda_\prec(\mathsf A),
\end{equation}
where $\mathcal O$ is the set of the four orders just specified. These quantities will be used in Sections~\ref{sec:lower}--\ref{sec:main}, also when the row set is only a subset of $\mathcal R_d(k)$. For example, the unique nonzero column of $M(2,3)$ has lexicographic leading row $x_1^2x_2$; consequently its leading-row count is one.

\subsection{Spikes, residues, and basic counting functions}\label{subsec:prelim-spikes}

\begin{definition}\label{def:spike}
A monomial $x^e$ is a spike if $e_i=2^{u_i}-1$ for integers $u_i\geq0$. This is \cite[\S1.5, Definition~1.5.2]{WalkerWoodI}. Let $\mathcal S_k(d)$ count the degree-$d$ spikes. Set $\mathcal S_0(0)=1$, $\mathcal S_0(d)=0$ for $d\ne0$, and $\mathcal S_k(d)=0$ for $d<0$. These definitions are used in Sections~\ref{sec:upper}, \ref{sec:projection}, and \ref{sec:main}.
\end{definition}

\begin{example}\label{ex:strict-spikes}
The degree-seven spikes in three variables have exponent tuples given by the permutations of $(7,0,0)$ and $(3,3,1)$. There are three of each, so $\mathcal S_3(7)=6$. Zero exponents are permitted because $0=2^0-1$. In Section~\ref{sec:upper}, an upper bound for the spike indices will be denoted by $\mathfrak h_{k,d}=\lfloor\log_2(d+k)\rfloor$, and $\varsigma_j$ will count the indices $u_i$ equal to $j$. For $(3,3,1)$, the index tuple is $(2,2,1)$, so $\varsigma_2=2$, $\varsigma_1=1$, and all other multiplicities vanish.
\end{example}

For $t,n\geq0$, define the Euclidean quotient $Q_t(n)$ and remainder $r_t(n)$ by
\[
 n=2^{t+1}Q_t(n)+r_t(n),\qquad 0\leq r_t(n)<2^{t+1}.
\]
For an exponent tuple $a$, write $\mathfrak a_t(a)=\sum_i(a_i\bmod2^{t+1})$. The powers $2^t$ and $2^{t+1}$ always remain explicit. These conventions are used in Sections~\ref{sec:upper}, \ref{sec:lower}, and \ref{sec:projection}. For instance, $Q_1(5)=1$, $r_1(5)=1$, and $\mathfrak a_1((5,2))=1+2=3$.

\subsection{Column counts and the first-square rank}\label{subsec:prelim-counts}

Let $Z_t(k,n)$ be the number of monomials of degree $n$ killed by $\Sq^{2^t}$, with $Z_t(k,n)=0$ for $n<0$. The number of nonzero labelled columns of the full matrix is
\begin{equation}\label{eq:W}
 \mathcal W_k(d)=\sum_{2^t\leq d}\bigl(N_k(d-2^t)-Z_t(k,d-2^t)\bigr).
\end{equation}
Write $\partial=\Sq^1$ and $\rho_k(n)=\rank(\partial:\mathcal P_k^n\to\mathcal P_k^{n+1})$ for $n\geq0$, and set $\rho_k(-1)=0$. For $d>0$, put
\begin{equation}\label{eq:W1}
 \mathcal W_k^{(1)}(d)=\rho_k(d-1)+\sum_{\substack{t\geq1\\2^t\leq d}}\bigl(N_k(d-2^t)-Z_t(k,d-2^t)\bigr),
\end{equation}
and define
\begin{equation}\label{eq:delta1}
 \delta_1(k,d)=N_k(d-1)-Z_0(k,d-1)-\rho_k(d-1)=\mathcal W_k(d)-\mathcal W_k^{(1)}(d).
\end{equation}
At $d=0$, both $\mathcal W_k^{(1)}(d)$ and $\delta_1(k,d)$ are zero. The quantities in this paragraph are used in Sections~\ref{sec:upper}, \ref{sec:lower}, and \ref{sec:main}. For $k=2,d=3$, the displayed matrix gives $Z_0(2,2)=2$, $Z_1(2,1)=2$, $\mathcal W_2(3)=1$, $\rho_2(2)=1$, and $\delta_1(2,3)=0$.

Two further counting functions will be used in Section~\ref{sec:projection} and the tables of Section~\ref{sec:main}:
\[
 \mathsf{Even}_k(n)=\begin{cases}N_k(n/2),&n\geq0\text{ is even},\\0,&\text{otherwise},\end{cases}
 \qquad
 \mathcal S_s^{\ne1}(n)=[\zeta^n]\left(1+\sum_{u\geq2}\zeta^{2^u-1}\right)^s.
\]
The first counts exponent tuples with all entries even. The second counts ordered spike-exponent tuples whose entries are different from one; $\zeta$ is a formal indeterminate, and $[\zeta^n]$ means coefficient extraction. Set $\mathcal S_s^{\ne1}(n)=0$ for $n<0$; when $s=0$, its value is one for $n=0$ and zero otherwise. Thus $\mathsf{Even}_2(4)=3$, with tuples $(4,0),(2,2),(0,4)$, while $\mathcal S_2^{\ne1}(3)=2$, with tuples $(3,0),(0,3)$. These are explicit counts, without an assertion of linear independence.

\subsection{Support graphs, matchings, and pivot orders}\label{subsec:prelim-graphs}

\begin{definition}\label{def:support-structures}
For a finite matrix $\mathsf A$, its support graph $G(\mathsf A)$ is bipartite, with one vertex class its rows and the other its labelled columns, and an edge at each nonzero entry. Write $G(k,d)=G(M(k,d))$. Its support hypergraph has the rows as vertices and one labelled hyperedge $\Supp(c)$ for every nonzero column. The bipartite incidence graph of a column family retains just those column vertices and the rows incident to them. These constructions will be used in Section~\ref{sec:lower} and in the combined bounds of Section~\ref{sec:main}.
\end{definition}

A matching is a set of edges with no common endpoints. It is induced when its matched vertices have no edges other than the matching edges. It is uniquely restricted when the subgraph induced by its matched vertices has that matching as its unique perfect matching. Denote the largest sizes of these three types by $\nu(G)$, $\nu_{\mathrm{ind}}(G)$, and $\nu_{\mathrm{ur}}(G)$, respectively. The last definition and its alternating-cycle characterization are due to \cite[\S1 and \S2.2, Theorem~2.1]{GolumbicHirstLewenstein2001}. An alternating cycle uses matching and nonmatching edges alternately. A forest is an undirected graph without cycles. This terminology is used in Section~\ref{sec:lower}.

For a chosen set $\mathscr C$ of nonzero labelled columns, a pivot assignment chooses distinct rows $\pi(c)\in\Supp(c)$. Its order constraints are the arrows
\begin{equation}\label{eq:pivot-arrows}
 v\longrightarrow\pi(c),\qquad v\in\Supp(c)\setminus\{\pi(c)\},\quad c\in\mathscr C.
\end{equation}
We call the assignment acyclic when this directed graph has no directed cycle, and denote the largest possible $|\mathscr C|$ by $\tau(\mathsf A)$. Thus the phrase acyclic pivots means that a single row order can satisfy all the indicated strict inequalities; this equivalence will be proved in Section~\ref{sec:lower}. These quantities are used in Sections~\ref{sec:lower} and \ref{sec:main}.

For example, take the two columns $\Sq^2(x_1x_2^2)=x_1x_2^4$ and $\Sq^1(x_1x_2^3)=x_1x_2^4+x_1^2x_2^3$. On the ordered rows $(x_1x_2^4,x_1^2x_2^3)$ their matrix is $\left(\begin{smallmatrix}1&1\\0&1\end{smallmatrix}\right)$. Choosing the first and second rows as the respective pivots produces just the arrow from the first row to the second. The resulting matching is uniquely restricted. By comparison, the all-one $2\times2$ matrix has an alternating four-cycle and two perfect matchings, so neither perfect matching is uniquely restricted. The support graph of the nonzero part of $M(2,3)$ is a single column vertex joined to its two output rows, hence a tree.

The elementary bounds
\begin{equation}\label{eq:matching-rank}
 \nu_{\mathrm{ind}}(G(\mathsf A))\leq\rank\mathsf A\leq\nu(G(\mathsf A))
\end{equation}
follow because an induced matching gives a permutation submatrix, whereas a nonzero determinant term in a nonsingular minor gives a matching. The stronger uniquely restricted bound will be proved in Section~\ref{sec:lower}. For the finite optimization described there, $n_{\mathrm{row}}(\mathsf A)$ denotes the number of rows, $y_{c,v}\in\{0,1\}$ records a proposed pivot pair, $s_c\in\{0,1\}$ records whether column $c$ is chosen, and $H_v\in\{1,\ldots,n_{\mathrm{row}}(\mathsf A)\}$ is its row-order height. These are auxiliary optimization variables, not polynomial exponents.

\subsection{Selected and complete singleton families}\label{subsec:prelim-singletons}

Fix $t\geq0$. For $d\geq2^t$ and $r_t(d)=0$, let $\mathcal D_t(k,d)$ consist of degree-$d$ monomials all of whose exponents are divisible by $2^{t+1}$. When $0<r_t(d)<2^t$, let it consist of those with exactly one exponent congruent to $r_t(d)$ modulo $2^{t+1}$ and at least $2^t$, all other exponents being divisible by $2^{t+1}$. Put $\mathcal D_t(k,d)=\varnothing$ when $d<2^t$ or $r_t(d)\geq2^t$. Define $D_t(k,d)=|\mathcal D_t(k,d)|$ and
\[
 D_k^{\cup}(d)=\left|\bigcup_{2^t\leq d}\mathcal D_t(k,d)\right|,\qquad D_k^{\cup}(0)=0.
\]
These are selected output sets; Section~\ref{sec:lower} proves that their elements are singleton targets. The integer $\iota(d)=\nu_2(d+1)$, where $\nu_2$ is the exponent of two in a positive integer, is the first zero binary position of $d$ and is used in the nesting argument there. These sets and counts also occur in Section~\ref{sec:main}.

Let $c_t^{\mathrm{sing}}(k,d)$ count all labelled singleton columns in layer $t$, and let $\sigma_t(k,d)$ count their distinct targets. For $0\leq r_t(d)<2^t$, write
\[
 \kappa_t(d)=\min\{2^j:2^j>r_t(d),\ j\geq0\},\qquad
 \Lambda_t(d)=\frac{2^t}{\kappa_t(d)}.
\]
Thus $\kappa_t(d)$ separates the binary digits which a unique allocation must use from the smaller unused digits. These quantities are used in Section~\ref{sec:lower} and evaluated in Section~\ref{subsec:k5d29}. In that counting argument $\mathbf q,\mathbf w,\mathbf v$ denote respectively the quotient tuple, the selected-digit tuple after division by $\kappa_t(d)$, and the smaller-digit tuple. For a target exponent tuple $e$, the quotient tuple $\mathbf y$ has coordinates $y_i=(e_i-v_i)/(2\kappa_t(d))$ and residues $\eta_i=y_i\bmod\Lambda_t(d)$.

For $k=2,d=5,t=1$, one has $Q_1(5)=1$, $r_1(5)=1$, $\kappa_1(5)=2$, and $\Lambda_1(5)=1$. The selected set is $\{x_1^5,x_2^5\}$. The four monomials $x_1^3,x_1^2x_2,x_1x_2^2,x_2^3$ have singleton $\Sq^2$ images $x_1^5,x_1^4x_2,x_1x_2^4,x_2^5$, respectively. Hence the selected count is two, while both the complete labelled count and the complete distinct-target count are four. Within the first layer of $M(2,4)$, however, $\Sq^1(x_1^2x_2)=\Sq^1(x_1x_2^2)=x_1^2x_2^2$. These two distinct labels have the same output, so a labelled singleton count can exceed the corresponding distinct-target count even within one layer.

\subsection{Simplicial boundary coordinates}\label{subsec:prelim-simplicial}

For a finite set of $s$ vertices, let $\Delta^{s-1}$ be the full simplex, whose faces are all subsets of its vertex set. The augmented chain space $\mathsf C_{q-1}(\Delta^{s-1};\F)$ has basis its $q$-element subsets, and $\mathsf C_{-1}=\F$ has basis the empty subset. The boundary is
\[
 \mathsf d_{q-1}(I)=\sum_{v\in I}(I\setminus\{v\}).
\]
In characteristic two there are no orientation signs; for $q=1$ the map sends every vertex to the empty subset. These are the usual simplicial and augmentation conventions, as in \cite[\S2.1]{Hatcher2002}. They will be used in Section~\ref{sec:lower}; $\mathsf d$ is distinguished from the polynomial differential $\partial=\Sq^1$.

For $1\leq q\leq k$, define $\mathfrak n_{t,q}(d)=(d-(q+1)2^t)/2^{t+1}$. The associated congruence family is present only if this number is a nonnegative integer, equivalently
\begin{equation}\label{eq:q-degree}
 \mathfrak n_{t,q}(d)\in\NN,\qquad d=2^{t+1}\mathfrak n_{t,q}(d)+(q+1)2^t.
\end{equation}
For a weak composition $\mathbf a$ of $\mathfrak n_{t,q}(d)$ and a $q$-subset $I\subseteq[k]$, put
\begin{equation}\label{eq:q-source}
 g(\mathbf a;I)=\prod_{i\in I}x_i^{2^{t+1}a_i+2^t}\prod_{i\notin I}x_i^{2^{t+1}a_i}.
\end{equation}
Its degree is $d-2^t$. Let $\mathbf u=\mathbf a+\mathbf1_I$, so $|\mathbf u|=\mathfrak n_{t,q}(d)+q$ and $I\subseteq\vsupp(\mathbf u)$. For a $(q-1)$-subset $J\subseteq\vsupp(\mathbf u)$, put
\begin{equation}\label{eq:q-row}
 f_{\mathbf u,J}=\prod_{i\in J}x_i^{2^{t+1}(u_i-1)+2^t}\prod_{i\notin J}x_i^{2^{t+1}u_i}.
\end{equation}
Each exponent of $g$ has all binary digits below position $t$ zero; exactly the coordinates in $I$ have digit one at position $t$. A surviving allocation of total $2^t$ must therefore choose exactly one of those digits. Thus
\begin{equation}\label{eq:q-boundary}
 \Sq^{2^t}\bigl(g(\mathbf u-\mathbf1_I;I)\bigr)=\sum_{v\in I}f_{\mathbf u,I\setminus\{v\}}.
\end{equation}
These coordinates are used in Section~\ref{sec:lower}. There $B_{t,q}(k,d)$ denotes the rank of this entire column family, and $B_t(k,d)=\sum_{q=1}^kB_{t,q}(k,d)$. For $q\geq2$, $F_{t,q}(k,d)$ denotes the size of the specified incidence-forest family formed, in each component $\mathbf u$, by fixing one $(q-1)$-subset $I_0$ and retaining the columns $I_0\cup\{v\}$ with $v\notin I_0$. Their values are zero when \eqref{eq:q-degree} has no permissible solution. They are also evaluated in Section~\ref{sec:main}.

For $k=3,t=0,q=2,d=5$, one has $\mathfrak n_{0,2}(5)=1$. Choose $\mathbf u=(1,1,1)$. The source exponent tuples for $I=\{1,2\},\{1,3\},\{2,3\}$ are $(1,1,2),(1,2,1),(2,1,1)$, while the rows indexed by $J=\{1\},\{2\},\{3\}$ have exponents $(1,2,2),(2,1,2),(2,2,1)$. The component matrix is
\[
 \begin{pmatrix}1&1&0\\1&0&1\\0&1&1\end{pmatrix}.
\]
Its columns are the boundaries of the three edges of a triangle. The first two are independent, and their sum is the third, so its rank is two. Retaining the two edges containing vertex one gives an incidence tree, illustrating the family counted by $F_{0,2}$.

\subsection{A homotopy for the first square}\label{subsec:prelim-contraction}

For $|e|>0$, define a linear map of degree minus one by
\begin{equation}\label{eq:h-definition}
 h(x^e)=\begin{cases}x^{e-\eps_{i(e)}},&e_{i(e)}\text{ is even},\\0,&e_{i(e)}\text{ is odd},\end{cases}
 \qquad h(1)=0.
\end{equation}
It will be called a contracting homotopy because Section~\ref{sec:upper} proves $\partial h+h\partial=\operatorname{id}$ in positive degree. The exactness of $(\mathcal P_k,\partial)$ in positive degree is \cite[\S1.3, Proposition~1.3.5]{WalkerWoodI}; the formula above specifies the particular homotopy used here. Both $h$ and $\partial$ are used in Sections~\ref{sec:upper}, \ref{sec:projection}, and \ref{sec:main}.

For example, $h(x_1^2x_2)=x_1x_2$ and $h(x_1x_2^2)=0$. Moreover,
\[
 \partial h(x_1^2x_2)=x_1^2x_2+x_1x_2^2,
 \qquad h\partial(x_1^2x_2)=h(x_1^2x_2^2)=x_1x_2^2.
\]
Their sum is $x_1^2x_2$, since the repeated term cancels over $\F$.

\subsection{Evaluated Steenrod relations}\label{subsec:prelim-relations}

A homogeneous relation among the generating squares is a finitely supported tuple $\Theta=(\Theta_t)$ with total degree $D_\Theta>0$, where every nonzero $\Theta_t$ has degree $D_\Theta-2^t$ and
\begin{equation}\label{eq:generator-relation}
 \sum_t\Sq^{2^t}\Theta_t=0.
\end{equation}
For $d\geq D_\Theta$, its evaluation is the linear map $\Psi_{\Theta,d}:\mathcal P_k^{d-D_\Theta}\to V(k,d)$ defined by $\Psi_{\Theta,d}(f)=(\Theta_t(f))_t$; its image is zero for $d<D_\Theta$. Every image vector belongs to $\ker\Phi_{k,d}$. Let $V_0(k,d)$ be the coordinate subspace of $V(k,d)$ spanned by all zero-column labels. For a finite set $\mathscr R$ of such relations, put
\[
 \Rel_{\mathscr R}(k,d)=\sum_{\Theta\in\mathscr R}\im\Psi_{\Theta,d},\qquad
 \Delta_{\mathscr R}=\dim\bigl(V_0(k,d)+\Rel_{\mathscr R}(k,d)\bigr)-\dim V_0(k,d).
\]
Denote the relation $\partial^2=0$ by $\Theta^{(1)}$, with $\Theta^{(1)}_0=\partial$ and all other entries zero. Set
\begin{equation}\label{eq:additional-correction}
 \begin{aligned}
 V_1(k,d)&=V_0(k,d)+\im\Psi_{\Theta^{(1)},d},\\
 \delta_{\mathscr R}^{(1)}(k,d)&=\dim\bigl(V_1(k,d)+\Rel_{\mathscr R}(k,d)\bigr)-\dim V_1(k,d).
 \end{aligned}
\end{equation}
These definitions are used in Section~\ref{sec:upper} and the combined and degree-seven bounds in Section~\ref{sec:main}. They measure dimension increments after already known relations have been included.

The Adem relations yield
\begin{equation}\label{eq:Adem22}
 \Sq^2\Sq^2+\Sq^1\Sq^2\Sq^1=0;
\end{equation}
see \cite[\S3.1, Definition~3.1.1 and the subsequent examples, p.~34]{WalkerWoodI}. Here $D_\Theta=4$, $\Theta_1=\Sq^2$, and $\Theta_0=\Sq^2\Sq^1$. As a small evaluation, take $k=2,d=5,f=x_1$. Then $\Theta_1(f)=0$ and $\Theta_0(f)=x_1^4$. The resulting vector is the first-layer label $(0,x_1^4)$, which is already a zero-column label because $\Sq^1(x_1^4)=0$. It therefore contributes no dimension increment modulo $V_0(2,5)$. This explains why separate relation ranks cannot simply be added.

\subsection{Divided powers and dual-spike modules}\label{subsec:prelim-dual}

Let $\DP_k^d=(\mathcal P_k^d)^*$, with dual monomial basis $(x^e)^*=v_1^{(e_1)}\cdots v_k^{(e_k)}$, and let $\langle\xi,f\rangle$ denote evaluation. This is the divided-power notation of \cite[\S9.1, Definition~9.1.1]{WalkerWoodI}. Here $v_i^{(0)}=1$ and $v_i^{(1)}=v_i$. Multiplication satisfies $v_i^{(n)}v_i^{(n')}=\binom{n+n'}n v_i^{(n+n')}$, with the coefficient reduced modulo two. Factors involving disjoint variables multiply with coefficient one. The divided power of a sum satisfies
\[
 (v_i+v_j)^{(n)}=\sum_{n'+n''=n}v_i^{(n')}v_j^{(n'')};
\]
see \cite[\S9.1, Definition~9.1.2 and Proposition~9.1.3]{WalkerWoodI}. These conventions are used in Sections~\ref{sec:upper}, \ref{sec:projection}, and \ref{sec:main}. For example, $v_1^{(1)}v_1^{(1)}=0$, while $v_1^{(1)}v_2^{(2)}$ is the functional extracting the coefficient of $x_1x_2^2$.

The down squares $\Sq_*^\vartheta$ are transposes of the polynomial squares:
\begin{equation}\label{eq:dual-square}
 \Sq_*^\vartheta((x^e)^*)=
 \sum_{\substack{b\in\NN^k,\ |b|=\vartheta\\b_i\leq e_i\ (1\leq i\leq k)}}
 \left(\prod_i\binom{e_i-b_i}{b_i}\bmod2\right)(x^{e-b})^*.
\end{equation}
Indeed, the coefficient at $(x^a)^*$ equals $\langle(x^e)^*,\Sq^\vartheta(x^a)\rangle$, and \eqref{eq:cartan} gives precisely the displayed coefficient with $a=e-b$. This is \cite[\S9.3, Definition~9.3.1 and Propositions~9.3.6--9.3.7]{WalkerWoodI}. For example, $\Sq_*^1(v_1^{(2)}v_2^{(1)})=v_1^{(1)}v_2^{(1)}$ and $\Sq_*^1(v_1^{(1)}v_2^{(2)})=v_1^{(1)}v_2^{(1)}$.

Write $K_k(d)=\bigcap_{\vartheta>0}\ker\Sq_*^\vartheta$. A dual spike is $(x^e)^*$ for a spike $x^e$, and $J_k(d)$ is the $\F\GL(k,\F)$-submodule of $K_k(d)$ generated by all degree-$d$ dual spikes. These are the Steenrod kernel and dual-spike module of \cite[\S9.4, Definitions~9.4.1, 9.4.3, and 9.4.5]{WalkerWoodI}. Linear substitutions on divided powers are made in the degree-one variables $v_i$ and extended by the preceding divided-power formula; this is the transpose convention of \cite[\S9.2, Definitions~9.2.1--9.2.2 and Proposition~9.2.3]{WalkerWoodI}. For a chosen subspace $\mathcal U\subseteq K_k(d)$, put $\ell_k(d)=\dim\mathcal U$. All these symbols are used in Sections~\ref{sec:upper}, \ref{sec:projection}, and \ref{sec:main}.

For example,
\[
 v_1^{(1)}v_2^{(2)}+v_1^{(2)}v_2^{(1)}=(v_1+v_2)^{(3)}+v_1^{(3)}+v_2^{(3)}\in J_2(3).
\]
The two pure terms are dual spikes, and $(v_1+v_2)^{(3)}$ is the image of $v_1^{(3)}$ under an invertible linear substitution. This is the example in \cite[\S9.4, Example~9.4.6]{WalkerWoodI}. The two computations of $\Sq_*^1$ above also show explicitly why the mixed terms cancel under that square.

A polynomial is strongly spike-free when every linear substitution, including a singular substitution, produces a polynomial having no spike term. The image of these polynomials in $\Qkd$ is denoted by $SF_k(d)$. This is \cite[\S30.2, Definition~30.2.1]{WalkerWoodII}. The identifications $SF_k(d)\cong(K_k(d)/J_k(d))^*$ and $\Qkd/SF_k(d)\cong J_k(d)^*$, as vector spaces, are the duality of \cite[\S30.2, Proposition~30.2.2]{WalkerWoodII}. This notation is used in Section~\ref{subsec:divided-power-dual} to distinguish dual-spike information from information about the remaining quotient. For example, $x_1^2$ is hit and is strongly spike-free: every linear substitution takes it to a sum of squares of degree-one elements, and no degree-two monomial with all spike exponents occurs in that sum. Its cohit class is zero.

\subsection{Projection coordinates and their dual functionals}\label{subsec:prelim-projection}

For $d>0$, let $\mathcal T_d(k)$ be the monomials of degree $d$ whose first positive exponent is odd, and put $T_d(k)=\Span\mathcal T_d(k)$ and $p_d=h\partial$. Define $\overline M(k,d)$ to have row basis $\mathcal T_d(k)$ and columns $p_d\Sq^{2^t}(x^a)$ with $t\geq1$, $2^t\leq d$, and $|a|=d-2^t$. Section~\ref{sec:projection} proves that $p_d$ is a projection onto $T_d(k)$ with kernel $\im\partial$. These objects and the following counts are used in Sections~\ref{sec:projection} and \ref{sec:main}.

Let $\overline Z_t(k,n)$ count the zero projected columns $h\Sq^{2^t+1}(x^a)$ with $|a|=n$, for $t\geq1$, and set it to zero for $n<0$. The count of nonzero labelled projected columns is
\begin{equation}\label{eq:projected-column-count}
 \overline{\mathcal W}_k(d)=\sum_{\substack{t\geq1\\2^t\leq d}}\bigl(N_k(d-2^t)-\overline Z_t(k,d-2^t)\bigr).
\end{equation}
Let $z_k(d)$ and $\overline N_k(d)$ be the numbers of zero and nonzero rows of $\overline M(k,d)$, respectively, so $\overline N_k(d)=|\mathcal T_d(k)|-z_k(d)$. For $x^e\in\mathcal T_d(k)$, let $\operatorname{coeff}_e:T_d(k)\to\F$ extract its coefficient, and define
\begin{equation}\label{eq:dual-phi}
 \varphi_e=(x^e)^*+\sum_{\substack{j>i(e)\\e_j\in2\Z_{>0}}}(x^{e+\eps_{i(e)}-\eps_j})^*.
\end{equation}
Finally, $\mathcal L_k^{\mathrm{ord}}(d)$ is the sum, over all orderings of the variables, of the spans of the corresponding $\varphi_e$ indexed by projected zero rows. No equivariance of a single $p_d$ is assumed.

For $k=2,d=3$, one has $\mathcal T_3(2)=\{x_1^3,x_1x_2^2,x_2^3\}$ and $p_3(x_1^2x_2)=x_1x_2^2$. All higher projected columns are zero, so $z_2(3)=3$, $\overline N_2(3)=\overline{\mathcal W}_2(3)=0$, and $\varphi_{(1,2)}$ is the mixed degree-three dual element displayed above. At degree five, $p_5\Sq^2(x_1x_2^2)=x_1x_2^4=h\Sq^3(x_1x_2^2)$, giving a nonzero singleton projected column. These examples exhibit both possibilities without eliminating any first-layer columns.

\subsection{Two-row tableaux and products on disjoint variable pairs}\label{subsec:prelim-tableaux}

Fix integers $s,p$ with $1\leq p\leq\lfloor s/2\rfloor$. The two-row partition $(s-p,p)$ has $s-p$ boxes in its first row and $p$ in its second. A tableau fills these boxes bijectively with $[s]$; a tabloid remembers only the unordered entries in each row. The permutation module on these tabloids can therefore be identified with the vector space on $p$-subsets of $[s]$, by recording the second row. For a tableau $\mathfrak t$, write $\{\mathfrak t\}$ for its row tabloid. Its column group $\operatorname{Col}(\mathfrak t)$ permutes entries within columns, and its polytabloid is
\[
 \mathfrak e_{\mathfrak t}=\sum_{\sigma\in\operatorname{Col}(\mathfrak t)}\operatorname{sgn}(\sigma)\{\mathfrak t\sigma\}.
\]
Over $\F$ all signs equal one. The span of the polytabloids is the Specht module. A standard tableau has entries increasing along each row and down each column. These are the definitions of \cite[\S4, Definitions~4.1 and 4.3, p.~13; \S8, Definition~8.1, p.~29]{James1978}. The standard-basis theorem is \cite[\S8, Theorem~8.4 and Corollary~8.5]{James1978}. This terminology will be used in Section~\ref{sec:projection}, where independence is also proved directly.

Write the bottom row as $\mathsf B=\{\mathfrak b_1<\cdots<\mathfrak b_p\}\subseteq[s]$ and its complement as $\mathsf T=\{\mathfrak t_1<\cdots<\mathfrak t_{s-p}\}$. A standard two-row tableau is specified by $\mathfrak b_i\geq2i$ for every $i$. For $\Omega=\{o_1<\cdots<o_s\}\subseteq[k]$ and an exponent assignment $e_j\in\{0,3,7,15,\ldots\}$ outside $\Omega$, with $\sum_{j\notin\Omega}e_j=d-s-p$, define
\[
 \xi_{\Omega,e,\mathsf B}=
 \prod_{i=1}^p\left(v_{o_{\mathfrak t_i}}^{(1)}v_{o_{\mathfrak b_i}}^{(2)}+v_{o_{\mathfrak t_i}}^{(2)}v_{o_{\mathfrak b_i}}^{(1)}\right)
 \prod_{i=p+1}^{s-p}v_{o_{\mathfrak t_i}}^{(1)}
 \prod_{j\notin\Omega}v_j^{(e_j)}.
\]
The identification of a tabloid with the positions of exponent two turns this expression into its polytabloid, multiplied by the fixed outside divided-power monomial. All factors in different pairs involve disjoint variables. Together with all degree-$d$ coordinate dual spikes, these vectors form the family whose prescribed cardinality is
\begin{equation}\label{eq:paired-bound}
 \beta_k(d)=\mathcal S_k(d)+\sum_{s=2}^k\binom{k}{s}\sum_{p=1}^{\lfloor s/2\rfloor}\left(\binom{s}{p}-\binom{s}{p-1}\right)\mathcal S_{k-s}^{\ne1}(d-s-p).
\end{equation}
This family is used in Sections~\ref{sec:projection} and \ref{sec:main}. The independence of this family and its membership in $J_k(d)$ will be proved there, rather than included in the definition.

For shape $(2,1)$, take top row $(1,3)$ and bottom row $(2)$. Its column group exchanges $1$ and $2$, so its polytabloid is $\{2\}+\{1\}$ in the subset basis. With $\Omega=\{1,2,3\}$ and no outside variables, it corresponds to $(v_1^{(1)}v_2^{(2)}+v_1^{(2)}v_2^{(1)})v_3^{(1)}$ in degree four. In the smaller case $k=2,d=3$, formula \eqref{eq:paired-bound} gives $\beta_2(3)=2+(\binom21-\binom20)=3$, consisting of the two coordinate dual spikes and their mixed companion.

\subsection{Combined bounds and admissible-monomial comparisons}\label{subsec:prelim-combined}

The functions used to combine the estimates in Section~\ref{sec:main} are, for $d>0$,
\begin{align}
 \mathsf{LB}_{\mathrm{opt}}(k,d)&=\max\{\tau(M(k,d)),\rho_k(d-1)+\tau(\overline M(k,d))\},\label{eq:lower-opt}\\
 \mathsf{LB}_{\mathrm{eval}}(k,d)&=\max\{\lamstd(M(k,d)),\rho_k(d-1)+\lamstd(\overline M(k,d))\},\label{eq:lower-eval}\\
 \mathsf{UB}(k,d)&=\min\bigl\{N_k(d)-\max\{\ell_k(d),\beta_k(d)\},\nonumber\\
 &\hspace{2em}\mathcal W_k^{(1)}(d)-\delta_{\mathscr R}^{(1)}(k,d),\ \rho_k(d-1)+\min\{\overline N_k(d),\overline{\mathcal W}_k(d)\}\bigr\}.\label{eq:upper-combined}
\end{align}
They are all defined to be zero for $d=0$. Here $\mathcal U\subseteq K_k(d)$ and $\mathscr R$ are the chosen dual subspace and relation set specified above. Section~\ref{sec:main} proves that these functions bound the hit rank in the directions indicated by their subscripts. For $k=2,d=3$, with $\mathcal U=0$ and $\mathscr R=\varnothing$, the preceding examples give $\mathsf{LB}_{\mathrm{eval}}=\mathsf{LB}_{\mathrm{opt}}=\mathsf{UB}=1$.

For the fixed left weight order with exponent-lexicographic tie-breaking, a monomial is inadmissible if it is congruent modulo hits to a linear combination of smaller monomials; otherwise it is admissible. This convention is \cite[Definitions~2--3]{MothebeUys2015}. Let $\mathcal B_k(n)$ denote the set of admissible monomials of degree $n$ in that order. In Section~\ref{subsec:k5d29}, for a spike exponent $\theta\in\{0,1,3,7,15\}$, $\mathcal I_\theta$ denotes the set of five-variable monomials obtained by inserting $\theta$ in each possible position of each exponent tuple in $\mathcal B_4(29-\theta)$. These notations are used only in Sections~\ref{subsec:k5d29}--\ref{subsec:reproducibility}.

\section{Upper bounds for the hit rank}\label{sec:upper}
The notation in Section~\ref{sec:matrix} allows upper rank bounds to be obtained from three kinds of information. A zero row removes a possible coordinate direction from the image, a zero column removes a generator from the labelled domain, and a relation among the remaining columns reduces the possible rank further. We establish the relevant counts, replace the first-layer column count by the exact rank of $\Sq^1$, and then combine relation spaces by taking their sum. The final subsection explains the equivalent bounds supplied by dual annihilators.

\subsection{Spikes and zero rows}
The spike convention and the counting function $\mathcal S_k(d)$ were fixed in Section~\ref{subsec:prelim-spikes}. The following obstruction is classical; see \cite[\S1.5, Proposition~1.5.3]{WalkerWoodI}. Its coefficient formulation will also be needed when one spike exponent is held fixed in a later argument.

\begin{lemma}\label{lem:spike-not-term}
Let $g\in\Pk$ be a monomial and let $\vartheta>0$. No spike occurs with nonzero coefficient in $\Sq^{\vartheta}(g)$.
\end{lemma}
\begin{proof}
Write $g=x^a$ and suppose that a spike $x^e$, with $e_i=2^{u_i}-1$ for every $i$, occurs in its image. By the Cartan--Lucas formula \eqref{eq:cartan}, there must be an allocation $b$ satisfying $a_i+b_i=e_i$, $b_i\preceq a_i$, and $\sum_i b_i=\vartheta$. Since $\vartheta>0$, at least one $b_i$ is positive. Fix such an index. The inequalities $0<b_i\leq e_i=2^{u_i}-1$ imply $u_i>0$, and all binary digits of $b_i$ occur among positions $0,\ldots,u_i-1$. Subtraction from $2^{u_i}-1$, whose digits in these positions are all one, complements each of these digits. Thus $a_i=(2^{u_i}-1)-b_i$ has digit zero wherever $b_i$ has digit one. Since $b_i$ has at least one nonzero digit, the condition $b_i\preceq a_i$ is impossible. Hence no such allocation exists, and the coefficient of every spike is zero.
\end{proof}

The obstruction identifies some zero rows immediately. To prove that it identifies all of them, it suffices to find one generating-square column in which a prescribed nonspike occurs. This is the coordinate-dual version of \cite[\S9.4, Proposition~9.4.4]{WalkerWoodI}.

\begin{theorem}\label{thm:exact-zero-rows}
Let $k\geq1$, $d\geq0$, and $x^e\in\mathcal R_d(k)$. Its row in $M(k,d)$ is zero if and only if $x^e$ is a spike.
\end{theorem}
\begin{proof}
If $x^e$ is a spike, Lemma~\ref{lem:spike-not-term} shows that its coefficient is zero in every column $\Sq^{2^t}(x^a)$. Its row is therefore zero.

Conversely, suppose that $x^e$ is not a spike. Some exponent $e_i$ is not of the form $2^u-1$. Since zero has that form, this exponent is positive. If all digits below its highest nonzero binary digit were one, it would equal $2^u-1$ for a suitable $u$. Therefore it has a zero digit at some position $t$ below that highest digit. The zero digit and the presence of a higher nonzero digit give
\[
 e_i=2^{t+1}Q_t(e_i)+r_t(e_i),\qquad Q_t(e_i)\geq1,\qquad 0\leq r_t(e_i)<2^t.
\]
Subtracting $2^t$ requires a borrow from the part above position $t$ and yields
\[
 e_i-2^t=2^{t+1}\bigl(Q_t(e_i)-1\bigr)+2^t+r_t(e_i).
\]
The $t$th digit of $e_i-2^t$ is consequently one, and this exponent is nonnegative. Define the domain exponent tuple by $a=e-2^t\eps_i$. Then $|a|=d-2^t$, so $(t,x^a)$ is a column label of $M(k,d)$. The allocation $b=2^t\eps_i$ satisfies $b_i\preceq a_i$ by the preceding calculation and $b_j=0\preceq a_j$ for $j\ne i$. Its output is $a+b=e$. The allocation is uniquely determined by this output and the fixed domain tuple, so its coefficient cannot cancel with another term of the same column. Thus the row of $x^e$ is nonzero. For $d=0$, the only row is the spike $1$ and there are no columns, which also gives the assertion.
\end{proof}

The resulting upper rank bound can be evaluated by counting exponent multiplicities. Spike-generated subspaces are studied in \cite{Mothebe2008Spikes}; the following formula specifies the coordinate count needed here.

\begin{proposition}\label{prop:spike-count}
Let $k\geq1$, $d\geq0$, and use $\mathfrak h_{k,d}=\lfloor\log_2(d+k)\rfloor$ from Section~\ref{subsec:prelim-spikes}. Then
\begin{equation}\label{eq:spike-count}
 \mathcal S_k(d)=
 \sum_{\substack{\varsigma_0,\ldots,\varsigma_{\mathfrak h_{k,d}}\geq0\\
       \sum_{j=0}^{\mathfrak h_{k,d}}\varsigma_j=k,\ 
       \sum_{j=0}^{\mathfrak h_{k,d}}2^j\varsigma_j=d+k}}
       \frac{k!}{\prod_{j=0}^{\mathfrak h_{k,d}}\varsigma_j!},
 \qquad
 \rank M(k,d)\leq N_k(d)-\mathcal S_k(d).
\end{equation}
\end{proposition}
\begin{proof}
A degree-$d$ spike has an ordered exponent tuple $(2^{u_1}-1,\ldots,2^{u_k}-1)$. Its degree condition is equivalent to $\sum_{i=1}^k2^{u_i}=d+k$. Every summand is positive and is at most the total, so $u_i\leq \mathfrak h_{k,d}$ for every $i$. For each binary position $j$, let $\varsigma_j$ be the number of indices $i$ for which $u_i=j$. Then the number of coordinates gives $\sum_j\varsigma_j=k$, and the degree condition gives $\sum_j2^j\varsigma_j=d+k$.

Conversely, any tuple of nonnegative multiplicities satisfying these two conditions determines the multiset of the $u_i$. Assigning this multiset to the $k$ labelled variable positions produces exactly $k!/\prod_j\varsigma_j!$ distinct exponent tuples. Different multiplicity tuples produce disjoint collections of spikes, so summing these multinomial coefficients gives the displayed count. By Theorem~\ref{thm:exact-zero-rows}, exactly $\mathcal S_k(d)$ of the $N_k(d)$ rows of the matrix vanish. Its image is contained in the coordinate space on the remaining rows, whose dimension is $N_k(d)-\mathcal S_k(d)$. This proves the rank bound.
\end{proof}

\subsection{The exact number of zero columns}
A column vanishes precisely when no admissible binary allocation has the required total. The next elementary selection lemma reduces this existence question to a comparison of integers.

\begin{lemma}\label{lem:binary-selection}
Let $t\geq0$ and let $\gamma_0,\ldots,\gamma_t$ be nonnegative integers. If $\sum_{j=0}^t2^j\gamma_j\geq2^t$, then there are integers $0\leq\chi_j\leq\gamma_j$ such that $\sum_{j=0}^t2^j\chi_j=2^t$.
\end{lemma}
\begin{proof}
Consider the finite set of integers $\sum_{j=0}^t2^j\chi_j$ obtained by allowing $0\leq\chi_j\leq\gamma_j$. Its subset consisting of integers at least $2^t$ is nonempty, because the choice $\chi_j=\gamma_j$ belongs to it. Choose a least member of this subset and denote it by $\mathfrak s$. Suppose that $\mathfrak s>2^t$. Let $j_0$ be the least index for which the chosen $\chi_{j_0}$ is positive. Decreasing that multiplicity by one is permitted. By the minimality of $\mathfrak s$, the resulting integer $\mathfrak s-2^{j_0}$ is less than $2^t$. Therefore
\[
 0<\mathfrak s-2^t<2^{j_0}.
\]
Every nonzero summand of $\mathfrak s$ has index at least $j_0$, so $2^{j_0}$ divides $\mathfrak s$. It also divides $2^t$, since $j_0\leq t$. Their positive difference cannot be strictly smaller than $2^{j_0}$. This contradiction proves $\mathfrak s=2^t$, as required.
\end{proof}

\begin{lemma}\label{lem:annihilation}
Let $t\geq0$ and let $a\in\NN^k$. Then
\begin{equation}\label{eq:annihilation}
 \Sq^{2^t}(x^a)=0
 \quad\Longleftrightarrow\quad
 \sum_{i=1}^k r_t(a_i)<2^t.
\end{equation}
\end{lemma}
\begin{proof}
An allocation in \eqref{eq:cartan} for the square $\Sq^{2^t}$ satisfies $0\leq b_i\leq2^t$. Thus it uses only binary digits in positions $0,\ldots,t$. These are exactly the digits recorded by the residue $r_t(a_i)=a_i\bmod2^{t+1}$, and
\[
 b_i\preceq a_i\quad\Longleftrightarrow\quad b_i\preceq r_t(a_i).
\]
In particular, $b_i\leq r_t(a_i)$. If $\sum_i r_t(a_i)<2^t$, then every permitted allocation has total less than $2^t$, and the image is zero.

For the converse, suppose that $\sum_i r_t(a_i)\geq2^t$. For each position $j\in\{0,\ldots,t\}$, let $\gamma_j$ count the variables whose residue $r_t(a_i)$ has digit one in position $j$. Then $\sum_j2^j\gamma_j=\sum_i r_t(a_i)$. Lemma~\ref{lem:binary-selection} supplies multiplicities $\chi_j\leq\gamma_j$ with weighted sum $2^t$. At position $j$, select any $\chi_j$ of the variables having the required digit one. For each variable $i$, let $b_i$ be the sum of the powers $2^j$ selected at that variable. No position is selected twice at one variable, so $b_i\preceq r_t(a_i)\preceq a_i$. The sum of all $b_i$ is $\sum_j2^j\chi_j=2^t$. Hence this allocation contributes a monomial to the image. Distinct allocations have distinct exponent tuples $a+b$, so this contribution cannot cancel. The image is therefore nonzero.
\end{proof}

The zero-column count $Z_t(k,n)$, introduced in Section~\ref{subsec:prelim-counts}, follows by separating quotient and residue coordinates in each domain exponent.

\begin{proposition}\label{prop:Z}
Let $n\geq0$ and $t\geq0$. With $n=2^{t+1}Q_t(n)+r_t(n)$ and $0\leq r_t(n)<2^{t+1}$, one has
\begin{equation}\label{eq:Z-formula}
 Z_t(k,n)=
 \begin{cases}
 \displaystyle\binom{Q_t(n)+k-1}{k-1}\binom{r_t(n)+k-1}{k-1},&r_t(n)<2^t,\\[5pt]
 0,&r_t(n)\geq2^t.
 \end{cases}
\end{equation}
\end{proposition}
\begin{proof}
Each domain exponent has the unique division $a_i=2^{t+1}Q_t(a_i)+r_t(a_i)$. If $x^a$ is annihilated by $\Sq^{2^t}$, Lemma~\ref{lem:annihilation} gives $0\leq\sum_i r_t(a_i)<2^t<2^{t+1}$. On the other hand, summing the exponent divisions gives
\[
 n=2^{t+1}\sum_{i=1}^k Q_t(a_i)+\sum_{i=1}^k r_t(a_i).
\]
Uniqueness of the quotient and remainder upon division of $n$ by $2^{t+1}$ now forces $\sum_i r_t(a_i)=r_t(n)$ and $\sum_iQ_t(a_i)=Q_t(n)$. In particular, no annihilated monomial exists if $r_t(n)\geq2^t$.

Suppose next that $r_t(n)<2^t$. Choose any weak composition of $Q_t(n)$ into $k$ quotient coordinates and any weak composition of $r_t(n)$ into $k$ residue coordinates. Every chosen residue coordinate is at most $r_t(n)$ and hence is strictly smaller than $2^{t+1}$, so these choices determine a unique exponent tuple by the displayed division formula. Its degree is $n$, and the sum of its residues is less than $2^t$, so its square is zero. This construction reverses the preceding assignment and is therefore a bijection. The two independent weak-composition counts are $\binom{Q_t(n)+k-1}{k-1}$ and $\binom{r_t(n)+k-1}{k-1}$, proving the formula.
\end{proof}

The count $\mathcal W_k(d)$ in \eqref{eq:W} is the total number of nonzero labelled columns. Deleting zero columns leaves this many vectors spanning the image, and therefore $\rank M(k,d)\leq\mathcal W_k(d)$. For example, $M(2,3)$ has only one nonzero column, namely $\Sq^1(x_1x_2)=x_1^2x_2+x_1x_2^2$, so $\mathcal W_2(3)=1$. It is necessary to count the individual monomial coefficients. The sum of those coefficients alone does not detect a zero polynomial: $\Sq^2(x_1x_2^3)=x_1x_2^5+x_1^2x_2^4$ is nonzero although its two coefficients have sum zero in $\F$ and $\binom{4}{2}$ is even.

\subsection{The first square and its contracting homotopy}
The first-layer column count can be improved because the rank of $\partial=\Sq^1$ is determined by exactness in positive degree; see \cite[\S1.3, Proposition~1.3.5]{WalkerWoodI}. We verify the explicit homotopy from \eqref{eq:h-definition}, since the same map will provide the quotient coordinates in Section~\ref{sec:projection}.

\begin{proposition}\label{prop:contraction}
On the positive-degree part of $\Pk$, one has $\partial h+h\partial=\operatorname{id}$. The ranks $\rho_k(n)$ defined in Section~\ref{subsec:prelim-counts} satisfy $\rho_k(0)=0$ and
\begin{equation}\label{eq:rho-recurrence}
 \rho_k(n)+\rho_k(n-1)=N_k(n)\quad(n\geq1),\qquad
 \sum_{n\geq0}\rho_k(n)\zeta^n=
       \frac{(1-\zeta)^{-k}-1}{1+\zeta}.
\end{equation}
\end{proposition}
\begin{proof}
The one-variable formula is $\partial(x_i^{a_i})=0$ when $a_i$ is even and $\partial(x_i^{a_i})=x_i^{a_i+1}$ when $a_i$ is odd. Fix a positive-degree monomial $x^a$ and let $i=i(a)$ be its first positive exponent position. Write $x^a=x_i^{a_i}g$, where $g$ involves only the variables with indices larger than $i$.

Suppose first that $a_i$ is even. Then $a_i\geq2$, so $h(x^a)=x_i^{a_i-1}g$ still has first positive position $i$. By the Cartan formula,
\[
 \partial h(x^a)=x_i^{a_i}g+x_i^{a_i-1}\partial g.
\]
The first factor of $x^a$ has even exponent, so $\partial(x^a)=x_i^{a_i}\partial g$. Every nonzero monomial term in this expression still has first positive position $i$ and exponent $a_i$ there. Applying $h$ term by term consequently gives $h\partial(x^a)=x_i^{a_i-1}\partial g$. The two occurrences of this polynomial cancel over $\F$, leaving $(\partial h+h\partial)(x^a)=x^a$.

Suppose instead that $a_i$ is odd. Now $h(x^a)=0$, while
\[
 \partial(x^a)=x_i^{a_i+1}g+x_i^{a_i}\partial g.
\]
The first summand has even positive exponent $a_i+1$ at position $i$, and its image under $h$ is $x_i^{a_i}g=x^a$. Every monomial term in the second summand has odd exponent $a_i$ at the same first positive position, so its image under $h$ is zero. Thus the required identity also holds in this case. Linearity proves it on every positive homogeneous component.

The Adem relation $\partial^2=0$ implies $\im\partial\subseteq\ker\partial$. Conversely, if $f\in\mathcal P_k^n$, $n>0$, and $\partial f=0$, the homotopy identity gives $f=\partial h(f)$. Hence
\[
 \ker\bigl(\partial:\mathcal P_k^n\longrightarrow\mathcal P_k^{n+1}\bigr)
 =\im\bigl(\partial:\mathcal P_k^{n-1}\longrightarrow\mathcal P_k^n\bigr).
\]
The kernel has dimension $\rho_k(n-1)$. Rank-nullity in the $N_k(n)$-dimensional domain therefore gives $\rho_k(n)+\rho_k(n-1)=N_k(n)$. Since $\partial(1)=0$, one has $\rho_k(0)=0$.

Multiply the recurrence by $\zeta^n$ and sum for $n\geq1$. The first rank sum is $\sum_{n\geq0}\rho_k(n)\zeta^n$, because its omitted constant term is zero. In the second sum, changing the index from $n$ to $n-1$ gives $\zeta\sum_{n\geq0}\rho_k(n)\zeta^n$. Finally, the monomial-counting identity $\sum_{n\geq0}N_k(n)\zeta^n=(1-\zeta)^{-k}$ shows that the right-hand side is $(1-\zeta)^{-k}-1$. Dividing the resulting identity by $1+\zeta$ in the ring of formal power series proves the asserted generating function.
\end{proof}

The quantity $\mathcal W_k^{(1)}(d)$ in \eqref{eq:W1} replaces the number of nonzero first-layer columns by their exact rank. The image of $M(k,d)$ is the sum of the layer images. The first has dimension $\rho_k(d-1)$, and each later image has dimension at most its number of nonzero columns. Therefore
\begin{equation}\label{eq:W1-bound}
 \rank M(k,d)\leq\mathcal W_k^{(1)}(d)\leq\mathcal W_k(d).
\end{equation}
The second inequality holds because a matrix rank cannot exceed its nonzero column count. Their difference is the nonnegative first-layer correction $\delta_1(k,d)$ in \eqref{eq:delta1}.

\subsection{Relations among the generating squares}
Relations must be counted in the labelled domain $V(k,d)$, since distinct labels can represent the same image vector. Section~\ref{subsec:prelim-relations} defined the evaluation maps $\Psi_{\Theta,d}$, the coordinate zero-column space $V_0(k,d)$, and the relation span $\Rel_{\mathscr R}(k,d)$. It also defined $V_1(k,d)$ by adjoining the evaluated relation $\partial^2=0$. The correction $\delta_{\mathscr R}^{(1)}(k,d)$ measures only the further dimension contributed modulo that space.

\begin{theorem}\label{thm:Adem-upper}
Let $k\geq1$, $d\geq0$, and let $\mathscr R$ be a finite set of homogeneous relations satisfying \eqref{eq:generator-relation}. Then
\begin{equation}\label{eq:Adem-upper}
 \rank M(k,d)\leq
 \mathcal W_k^{(1)}(d)-\delta_{\mathscr R}^{(1)}(k,d)
 \leq\min\{\mathcal W_k^{(1)}(d),
               \mathcal W_k(d)-\Delta_{\mathscr R}\}.
\end{equation}
\end{theorem}
\begin{proof}
First assume $d\geq2$. The domain of the first layer is $\mathcal P_k^{d-1}$, and its column map is $\partial$. Since $d-1>0$, Proposition~\ref{prop:contraction} gives
\[
 \ker\bigl(\partial:\mathcal P_k^{d-1}\longrightarrow\mathcal P_k^d\bigr)
 =\im\bigl(\partial:\mathcal P_k^{d-2}\longrightarrow\mathcal P_k^{d-1}\bigr).
\]
The right-hand side is precisely the first-layer component of $\im\Psi_{\Theta^{(1)},d}$, and this evaluation has no components in any other layer. Consequently the evaluated relation supplies the entire first-layer kernel. That kernel has dimension $N_k(d-1)-\rho_k(d-1)$ and already contains the $Z_0(k,d-1)$ coordinate vectors corresponding to the zero columns in that layer. The zero-column coordinate vectors from other layers lie in different direct summands of $V(k,d)$, so they have zero intersection with this kernel. Hence
\[
 \dim V_1(k,d)-\dim V_0(k,d)
 =N_k(d-1)-\rho_k(d-1)-Z_0(k,d-1)
 =\delta_1(k,d).
\]
For $d=1$, the only domain summand is $\mathcal P_k^0$, and its sole column is $\partial(1)=0$. Thus $V_1(k,1)=V_0(k,1)=V(k,1)$ and the same identity holds with both sides zero. For $d=0$, the labelled domain and both subspaces are zero, so the identity holds there as well.

Every vector of $V_0(k,d)$ maps to zero under $\Phi_{k,d}$ by definition. For an evaluated relation and a domain polynomial $f$, one has
\[
 \Phi_{k,d}\Psi_{\Theta,d}(f)
 =\sum_t\Sq^{2^t}\Theta_t(f)=0
\]
by \eqref{eq:generator-relation}. It follows that $V_1(k,d)+\Rel_{\mathscr R}(k,d)$ is contained in $\ker\Phi_{k,d}$. Rank-nullity, the definition of $\delta_{\mathscr R}^{(1)}$, and the preceding dimension identity give
\begin{align*}
 \rank M(k,d)
 &\leq c_k^{\rm tot}(d)-\dim\bigl(V_1(k,d)+\Rel_{\mathscr R}(k,d)\bigr)\\
 &=c_k^{\rm tot}(d)-\dim V_1(k,d)-\delta_{\mathscr R}^{(1)}(k,d)\\
 &=c_k^{\rm tot}(d)-\dim V_0(k,d)-\delta_1(k,d)-\delta_{\mathscr R}^{(1)}(k,d)\\
 &=\mathcal W_k(d)-\delta_1(k,d)-\delta_{\mathscr R}^{(1)}(k,d)\\
 &=\mathcal W_k^{(1)}(d)-\delta_{\mathscr R}^{(1)}(k,d).
\end{align*}
Here $c_k^{\rm tot}(d)-\dim V_0(k,d)=\mathcal W_k(d)$, because the coordinate vectors spanning $V_0(k,d)$ correspond exactly to the zero columns.

Adjoining a subspace cannot decrease dimension, so $\delta_{\mathscr R}^{(1)}(k,d)\geq0$. This proves the comparison with $\mathcal W_k^{(1)}(d)$. Moreover, $V_0(k,d)\subseteq V_1(k,d)$ gives
\[
 V_0(k,d)+\Rel_{\mathscr R}(k,d)
 \subseteq V_1(k,d)+\Rel_{\mathscr R}(k,d).
\]
Subtracting their dimensions from $c_k^{\rm tot}(d)$ shows that the first upper bound is at most
\[
 c_k^{\rm tot}(d)-\dim\bigl(V_0(k,d)+\Rel_{\mathscr R}(k,d)\bigr)
 =\mathcal W_k(d)-\Delta_{\mathscr R},
\]
which completes the proof.
\end{proof}

Thus the correction is a rank increment after adjoining all relation vectors to a common matrix. Separate relation ranks cannot simply be added, since the relation images can intersect each other or the zero-column space. In particular, including $\partial^2=0$ again in $\mathscr R$ contributes no additional correction. The concrete relation \eqref{eq:Adem22} follows from $\Sq^2\Sq^2=\Sq^3\Sq^1$ and $\Sq^1\Sq^2=\Sq^3$; these identities occur in \cite[\S3.1, examples following Definition~3.1.1]{WalkerWoodI}. Its evaluated contribution is calculated in Section~\ref{subsec:k3d7}.

\subsection{Dual annihilators and upper rank bounds}\label{subsec:divided-power-dual}
The degreewise divided-power pairing from Section~\ref{subsec:prelim-dual} converts independent annihilating functionals into upper rank bounds. This is the dimension statement underlying \cite[\S9.4, Proposition~9.4.2]{WalkerWoodI}; we include the argument to specify exactly how the chosen subspace is used.

\begin{proposition}\label{prop:dual-certificate}
Let $\mathcal U\subseteq K_k(d)$ have dimension $\ell_k(d)$. Then
\begin{equation}\label{eq:dual-upper}
 \rank M(k,d)\leq N_k(d)-\ell_k(d).
\end{equation}
\end{proposition}
\begin{proof}
For $\xi\in\DP_k^d$, $1\leq\vartheta\leq d$, and a homogeneous polynomial $f$ of degree $d-\vartheta$, the definition of a transpose square gives
\[
 \langle\Sq_*^{\vartheta}\xi,f\rangle=\langle\xi,\Sq^{\vartheta}(f)\rangle.
\]
If $\xi\in K_k(d)$, the left-hand side vanishes for every $\vartheta>0$ and every such $f$. Hence $\xi$ vanishes on all positive-square images and therefore on the hit subspace. Conversely, if $\xi$ vanishes on the hit subspace, the right-hand side is zero for every $\vartheta>0$ and every $f$. The pairing in degree $d-\vartheta$ is nondegenerate, so $\Sq_*^{\vartheta}\xi=0$ for every $\vartheta>0$, and consequently $\xi\in K_k(d)$. Squares of degree greater than $d$ vanish on $\DP_k^d$ for degree reasons. Thus $K_k(d)$ is exactly the annihilator of $(\Aplus\Pk)_d$; this also covers $d=0$, when every positive transpose square vanishes and the hit subspace is zero.

For completeness, choose a basis of the hit subspace and extend it to a basis of $\mathcal P_k^d$. A functional lies in its annihilator precisely when it vanishes on the first group of basis elements; its values on all remaining basis elements can be chosen independently. Since the ambient dimension is $N_k(d)$, the annihilator therefore has dimension $N_k(d)-\dim(\Aplus\Pk)_d$. The matrix image is the hit subspace by \eqref{eq:dimension-rank}, so
\[
 \dim K_k(d)=N_k(d)-\rank M(k,d).
\]
Finally, $\mathcal U\subseteq K_k(d)$ implies $\ell_k(d)\leq\dim K_k(d)$. Rearranging gives the stated bound.
\end{proof}

The coordinate spike functionals are distinct dual basis vectors and belong to $K_k(d)$ by Theorem~\ref{thm:exact-zero-rows}. Their span therefore gives $\ell_k(d)=\mathcal S_k(d)$. The dual-spike module $J_k(d)$ can contain further independent vectors; its definition and a degree-three example are in \cite[\S9.4, Definition~9.4.5 and Example~9.4.6]{WalkerWoodI}. Further representation-theoretic constructions are developed in \cite[Chapters~23--24]{WalkerWoodII}.

The distinction between zero rows and general annihilators also explains the role of the strongly spike-free subspace recalled in Section~\ref{subsec:prelim-dual}. The identities $SF_k(d)\cong(K_k(d)/J_k(d))^*$ and $\Qkd/SF_k(d)\cong J_k(d)^*$ concern dual subquotients, with the transpose action in the module formulation; see \cite[\S30.2, Definition~30.2.1 and Proposition~30.2.2, p.~315]{WalkerWoodII}. They do not assert that additional individual rows of the original matrix vanish. Similarly, independent admissible cohit classes need not be represented by zero-row monomials. For example, $M(2,3)$ has image spanned by $x_1^2x_2+x_1x_2^2$, so neither $x_1^2x_2$ nor $x_1x_2^2$ belongs to that one-dimensional image. Each has a nonzero cohit class although each occurs in the nonzero first-layer column. This is also the degree-three example in \cite[\S1.5, Example~1.5.4, p.~12]{WalkerWoodI}.

\section{Lower bounds from matrix minors}\label{sec:lower}

The preceding section bounds the dimension of the hit space from above by identifying coordinates and relations that cannot contribute to its rank. We now obtain lower bounds by exhibiting independent columns. The definitions in Sections~\ref{subsec:prelim-singletons}--\ref{subsec:prelim-simplicial} distinguish singleton columns, triangular submatrices, incidence forests, and simplicial boundary matrices. After calculating their contributions, we compare the resulting bounds so that overlapping constructions are not counted twice.

\subsection{An explicit singleton family}\label{subsec:selected-singletons}

A column with one nonzero entry immediately provides a pivot, independently of the chosen row order. The congruence conditions defining $\mathcal D_t(k,d)$ were selected to make the corresponding Steenrod expansion have exactly this form. Indeed, if the $t$th binary digit of a domain exponent $a_i$ is one and $2^{t+1}\mid a_j$ for every $j\ne i$, then every admissible allocation of total $2^t$ vanishes outside position $i$. The remaining allocation must equal $2^t$ at that position, and Lucas' criterion permits it. Therefore
\begin{equation}\label{eq:singleton-source}
 \Sq^{2^t}(x^a)=x^{a+2^t\eps_i}.
\end{equation}

The following count concerns the selected output set $\mathcal D_t(k,d)$, rather than all columns whose support has one element. The distinction is needed because different columns can have the same output and some singleton outputs do not satisfy the selected congruence conditions.

\begin{proposition}\label{prop:D}
Let $k\geq1$, $t\geq0$, and $d\geq2^t$. In the degree division $d=2^{t+1}Q_t(d)+r_t(d)$ from Section~\ref{subsec:prelim-spikes}, one has
\begin{equation}\label{eq:D}
 D_t(k,d)=
 \begin{cases}
 \displaystyle\binom{Q_t(d)+k-1}{k-1},&r_t(d)=0,\\[5pt]
 \displaystyle k\binom{Q_t(d)+k-2}{k-1},&0<r_t(d)<2^t,\\[5pt]
 0,&2^t\leq r_t(d)<2^{t+1}.
 \end{cases}
\end{equation}
Every monomial in $\mathcal D_t(k,d)$ is the output of a singleton column. Moreover, every singleton column in the $t$th layer, whether or not its output belongs to $\mathcal D_t(k,d)$, satisfies $r_t(d)<2^t$.
\end{proposition}

\begin{proof}
Suppose first that $r_t(d)=0$. An output in the selected set has exponents $e_i=2^{t+1}w_i$ for $1\leq i\leq k$, where the nonnegative integers $w_i$ sum to $Q_t(d)$. There are $\binom{Q_t(d)+k-1}{k-1}$ such weak compositions. Since $d\geq2^t>0$, at least one $w_i$ is positive. Subtract $2^t$ from the corresponding exponent. Its residue modulo $2^{t+1}$ becomes $2^t$, while every other exponent retains residue zero. The resulting monomial has degree $d-2^t$, and \eqref{eq:singleton-source} maps it to the chosen output. Thus every monomial counted in the first case is realized by a singleton column.

Suppose next that $0<r_t(d)<2^t$. Exactly one exponent is congruent to $r_t(d)$ modulo $2^{t+1}$, and the definition of the selected set requires that exponent to be at least $2^t$. Write it as $r_t(d)+2^{t+1}w_i$. Since $r_t(d)<2^t$, the lower bound is equivalent to $w_i\geq1$. Every other exponent is $2^{t+1}w_j$, and $\sum_jw_j=Q_t(d)$. After choosing the distinguished index $i$, replace $w_i$ by $w_i-1$. This is a bijection with weak compositions of $Q_t(d)-1$ into $k$ parts. The distinguished index is determined by the nonzero residue, so these $k$ families are disjoint. This proves the second count. Subtracting $2^t$ from its distinguished exponent gives residue $2^t+r_t(d)$ modulo $2^{t+1}$, whose $t$th digit is one. Equation~\eqref{eq:singleton-source} again proves that the monomial is a singleton output. The remaining case has an empty selected set by definition.

It remains to prove the necessary condition for an arbitrary singleton column. Let its domain be $x^a$, and put $\bar a_i=a_i\bmod2^{t+1}$. Only binary digits of weights at most $2^t$ can enter an allocation of total $2^t$, so the possible allocations are exactly the selections from the labelled binary digits of the $\bar a_i$ having that total weight. Nonvanishing gives $\sum_i\bar a_i\geq2^t$. If this sum were at least $2^{t+1}$, remove the digits used by the unique allocation. The remaining digits would still have total weight at least $2^t$, and Lemma~\ref{lem:binary-selection} would select from them a second collection of weight $2^t$. It would be disjoint from the first selection and would therefore give a different allocation and a different output monomial. This contradicts the singleton hypothesis. Consequently
\[
 2^t\leq\sum_i\bar a_i<2^{t+1}.
\]
Since $|a|=d-2^t$, the sum on the left is congruent to $d-2^t$ modulo $2^{t+1}$. It follows that $r_t(d)=\sum_i\bar a_i-2^t<2^t$, as claimed.
\end{proof}

The necessary residue condition does not characterize the selected family. For example,
\begin{equation}\label{eq:singleton-outside}
 \Sq^4(x_1^2x_2^2x_3)=x_1^4x_2^4x_3.
\end{equation}
The only available digits of total weight four are the two digits of weight two in the first two exponents. Nevertheless, the output $(4,4,1)$ belongs to none of the sets $\mathcal D_t(3,9)$: the first layer has no selected output in odd degree, the second layer would require the exponent congruent to one modulo four to be at least two, and the third layer would require the other exponents to be divisible by eight. The selected sets themselves have a useful nesting property.

\begin{proposition}\label{prop:nested-singletons}
Let $d>0$, and let $\iota(d)=\nu_2(d+1)$. If $2^{\iota(d)}>d$, then every active selected set $\mathcal D_t(k,d)$ is empty. Otherwise
\begin{equation}\label{eq:nested-singletons}
 \bigcup_{2^t\leq d}\mathcal D_t(k,d)=\mathcal D_{\iota(d)}(k,d),
 \qquad D_k^{\cup}(d)=\max_{2^t\leq d}D_t(k,d).
\end{equation}
Every total row order satisfies $D_k^{\cup}(d)\leq\lambda_\prec(M(k,d))$. At $d=0$, the convention is $D_k^{\cup}(0)=0$.
\end{proposition}

\begin{proof}
The integer $\iota(d)$ is the position of the first zero in the binary expansion of $d$, starting with position zero. Indeed, adding one changes the initial string of ones to zeros, and $\nu_2(d+1)$ is its length. A nonempty selected set requires $r_t(d)<2^t$, which means that the $t$th binary digit of $d$ is zero. Hence every index contributing to the union satisfies $t\geq\iota(d)$. If $2^{\iota(d)}>d$, no such index is active.

Assume that $2^{\iota(d)}\leq d$. When $\iota(d)=0$, the degree $d$ is even. In every selected output, the nondistinguished exponents are divisible by $2^{t+1}$, and the possible distinguished residue is $r_t(d)$, which is also even. Thus all output exponents are even. This is precisely the condition for membership in $\mathcal D_0(k,d)$, so every selected set is contained in that set.

Now suppose that $\iota(d)>0$ and $t>\iota(d)$. The degree satisfies
\[
 d\equiv2^{\iota(d)}-1\pmod{2^{\iota(d)+1}}.
\]
In particular, $r_t(d)$ is nonzero. An output in $\mathcal D_t(k,d)$ has a unique distinguished exponent, while all other exponents are divisible by $2^{t+1}$ and hence by $2^{\iota(d)+1}$. The distinguished exponent is congruent to $2^{\iota(d)}-1$ modulo $2^{\iota(d)+1}$, and it is at least $2^t$, hence at least $2^{\iota(d)}$. These are the defining conditions for $\mathcal D_{\iota(d)}(k,d)$. This proves the containment for every contributing index. Since $t=\iota(d)$ itself is active, the union equals that largest member, proving both identities in \eqref{eq:nested-singletons}.

Finally, a singleton output is the leading row of its column for every total order. Distinct elements of the union therefore contribute distinct leading rows, giving the stated inequality.
\end{proof}

For even positive $d$, the union count is $\binom{d/2+k-1}{k-1}$. For odd $d$ that is not a spike degree, it is $k\binom{Q_{\iota(d)}(d)+k-2}{k-1}$. Thus taking the union avoids repeated outputs but does not improve the maximum of the selected layer counts.

\subsection{Exact counts for all singleton columns}\label{subsec:all-singletons}

The previous congruence condition gives a convenient subfamily. To count all singleton columns, we must identify exactly when the available binary digits admit one allocation of weight $2^t$. We then count column labels and output monomials separately, because the map from labels to outputs need not be injective.

\begin{theorem}[All singleton columns]\label{thm:all-singletons}
Fix $k\geq1$, $d\geq0$, and $t\geq0$. The counts $c_t^{\rm sing}(k,d)$ and $\sigma_t(k,d)$ defined in Section~\ref{subsec:prelim-singletons} are zero if $Q_t(d)=0$ or $r_t(d)\geq2^t$. In the remaining cases, the cutoff $\kappa_t(d)$ is at most $2^t$, and $\Lambda_t(d)=2^t/\kappa_t(d)$ is a positive integer. One has
\begin{equation}\label{eq:all-singleton-columns}
 c_t^{\rm sing}(k,d)
 =\binom{Q_t(d)+k-2}{k-1}\binom{\Lambda_t(d)+k-1}{k-1}\binom{r_t(d)+k-1}{k-1}.
\end{equation}
\begin{equation}\label{eq:all-singleton-targets}
\begin{split}
 \sigma_t(k,d)
 ={}&\binom{r_t(d)+k-1}{k-1}
 \Biggl\{\binom{Q_t(d)+k-1}{k-1}\\
 &\qquad+\left[\binom{\Lambda_t(d)+k-1}{k-1}-k\right]
 \binom{Q_t(d)+k-2}{k-1}\Biggr\}.
\end{split}
\end{equation}
The complete set of singleton outputs consists of the distinct monomials
\[
 \prod_{i=1}^k x_i^{2^{t+1}q_i+2\kappa_t(d)w_i+v_i},
 \qquad |\mathbf q|=Q_t(d)-1,\quad |\mathbf w|=\Lambda_t(d),\quad |\mathbf v|=r_t(d),
\]
where $\mathbf q,\mathbf w,\mathbf v\in\NN^k$. In particular, $D_t(k,d)\leq\sigma_t(k,d)$.
\end{theorem}

\begin{proof}
Let $x^a$ be a domain monomial, and again write $\bar a_i=a_i\bmod2^{t+1}$. Each binary digit of each $\bar a_i$ is regarded as a separate available object, retaining its variable index. An allocation $b$ with $|b|=2^t$ and $b_i\preceq a_i$ is equivalent to a selection of these objects whose weights sum to $2^t$. Two different selections give different allocation vectors: within a fixed variable, the binary expansion uniquely determines which digits were used. By \eqref{eq:cartan}, different allocation vectors produce different output exponents. Thus the column is a singleton exactly when there is one such selection.

If there is one selection, the last part of Proposition~\ref{prop:D} shows that $2^t\leq\sum_i\bar a_i<2^{t+1}$. Comparing this sum with the degree of the domain gives
\[
 \sum_i\bar a_i=2^t+r_t(d),\qquad 0\leq r_t(d)<2^t.
\]
We next characterize uniqueness within this range. Process the available digits in decreasing order of weight, taking as many as possible without making the selected sum exceed $2^t$. Before weight $2^j$ is considered, the unfilled amount is divisible by $2^j$, because all previous weights are divisible by $2^{j+1}$. If enough digits of weight $2^j$ are available to fill that amount, the process terminates at that weight. Otherwise it takes all of them and proceeds to the next smaller weight. Lemma~\ref{lem:binary-selection}, or this successive divisibility argument, ensures that the process eventually fills the amount.

Suppose that it terminates at weight $2^j$. All available larger digits have then been selected. If some digit of weight $2^j$ were selected and another of the same weight were not selected, exchanging them would give a second labelled selection. Uniqueness therefore requires every digit of weight $2^j$ to be selected. Moreover, the total of the smaller digits must be less than $2^j$. Otherwise Lemma~\ref{lem:binary-selection} would select smaller digits of total $2^j$, and they could replace one selected digit of that weight. We have proved that uniqueness implies the existence of a cutoff weight $2^j$ for which all digits of weight at least $2^j$ have total $2^t$ and all smaller digits have total $r_t(d)<2^j$.

Conversely, suppose such a cutoff exists. The selection of all digits at or above it has weight $2^t$. Any other selection must omit at least one of these digits, losing weight at least $2^j$. All smaller available digits together have weight less than $2^j$, so they cannot compensate for the loss. The original selection is therefore unique.

The least power of two strictly greater than $r_t(d)$ is $\kappa_t(d)$. The cutoff criterion just proved is equivalent to
\[
 \sum_i(\bar a_i\bmod\kappa_t(d))=r_t(d).
\]
To see the forward implication, the digits below a valid cutoff have total $r_t(d)<\kappa_t(d)$, so none can have weight at least $\kappa_t(d)$. Their total is consequently the sum of the indicated residues. Conversely, the residue identity says that the digits below $\kappa_t(d)$ have total $r_t(d)$, while all remaining digits have total $2^t$. Since the former total is smaller than $\kappa_t(d)$, the cutoff criterion applies.

If $Q_t(d)=0$, then $d<2^{t+1}$, so a domain monomial of degree $d-2^t$ either does not exist or has degree less than $2^t$. Instability makes its image zero. If $r_t(d)\geq2^t$, Proposition~\ref{prop:D} excludes a singleton column. We may therefore assume that $Q_t(d)\geq1$ and $r_t(d)<2^t$. Successive division of the exponents gives the unique expression
\[
 a_i=2^{t+1}q_i+\kappa_t(d)w_i+v_i,
 \qquad |\mathbf q|=Q_t(d)-1,\quad
 |\mathbf w|=\Lambda_t(d),\quad |\mathbf v|=r_t(d).
\]
The cutoff criterion gives $|\mathbf v|=r_t(d)$ and $\kappa_t(d)|\mathbf w|=2^t$; subtracting these contributions from the total domain degree gives the equation for $|\mathbf q|$. Conversely, every triple of weak compositions with these sums gives valid successive divisions: $v_i\leq r_t(d)<\kappa_t(d)$ and $w_i\leq\Lambda_t(d)<2\Lambda_t(d)=2^{t+1}/\kappa_t(d)$. Its unique allocation is $b_i=\kappa_t(d)w_i$. Indeed, this allocation uses exactly the digits at or above the cutoff, and their total is $2^t$. The output is
\[
 \Sq^{2^t}(x^a)=\prod_{i=1}^k x_i^{2^{t+1}q_i+2\kappa_t(d)w_i+v_i}.
\]
The three weak compositions can be chosen independently. Their numbers are the three factors of \eqref{eq:all-singleton-columns}, proving the labelled-column count.

For the output count, first observe that an output determines each $v_i$ from its exponent residue modulo $2\kappa_t(d)$. After removing these residues and dividing by $2\kappa_t(d)$, write the resulting tuple as $\mathbf y$. It satisfies
\[
 y_i=\Lambda_t(d)q_i+w_i,\qquad |\mathbf y|=\Lambda_t(d)Q_t(d).
\]
Put $\eta_i=y_i\bmod\Lambda_t(d)$; when $\Lambda_t(d)=1$, this residue is zero. In any representation above, $0\leq w_i\leq\Lambda_t(d)$ and $w_i\equiv\eta_i\pmod{\Lambda_t(d)}$. Thus $w_i=\eta_i$ unless $\eta_i=0$, in which case $w_i$ can also equal $\Lambda_t(d)$. Since $\sum_iw_i=\Lambda_t(d)$, exactly two cases are possible.

In the first case, $\sum_i\eta_i=0$. All $y_i$ are multiples of $\Lambda_t(d)$, and their quotients form a weak composition of $Q_t(d)$. There are $\binom{Q_t(d)+k-1}{k-1}$ such tuples. Each is realizable: because $Q_t(d)\geq1$, choose an index with positive quotient, set $w_i=\Lambda_t(d)$ there and $w_j=0$ elsewhere, and reduce that quotient by one to obtain $\mathbf q$.

In the second case, $\sum_i\eta_i=\Lambda_t(d)$. Necessarily $w_i=\eta_i$ for every $i$, and the tuple of quotients $q_i$ has sum $Q_t(d)-1$. The allowed residue tuples are the weak compositions of $\Lambda_t(d)$ in which every entry is strictly smaller than $\Lambda_t(d)$. The only excluded compositions have one entry equal to $\Lambda_t(d)$ and all others zero, giving exactly $k$ exclusions. Hence this case contributes
\[
 \left[\binom{\Lambda_t(d)+k-1}{k-1}-k\right]
 \binom{Q_t(d)+k-2}{k-1}
\]
different tuples $\mathbf y$. No further case can occur: the residue sum is a multiple of $\Lambda_t(d)$ because $|\mathbf y|$ is such a multiple, and it cannot exceed $\sum_iw_i=\Lambda_t(d)$. For $\Lambda_t(d)=1$, the second contribution is zero, as required.

The two cases are disjoint and, within each case, the residues and quotients reconstruct $\mathbf y$ uniquely. Finally, $\mathbf v$ has $\binom{r_t(d)+k-1}{k-1}$ choices and is recovered uniquely from the output. Multiplying by this number proves \eqref{eq:all-singleton-targets}. The selected set is a subset of the complete singleton-output set by Proposition~\ref{prop:D}, proving the final inequality.
\end{proof}

Every singleton output is a leading row for every total order. Consequently
\[
 D_t(k,d)\leq\sigma_t(k,d)\leq\lambda_\prec(M(k,d)).
\]
The exact count can strengthen the selected singleton estimate, but it is already included in a bound using all leading rows. Counts from different layers may be added only after verifying that their output sets are disjoint.

\subsection{Triangular minors and uniquely restricted matchings}\label{subsec:triangular-matchings}

Singleton columns are the simplest case of a triangular minor. More generally, the other entries of a selected column are harmless if its chosen pivot can be ordered after all of them. The graph formulation in Section~\ref{subsec:prelim-graphs} records precisely these ordering requirements.

\begin{lemma}\label{lem:triangular}
Let $\mathsf A$ be a finite matrix over $\F$ with an ordered row set. Nonzero columns with distinct leading rows are linearly independent.
\end{lemma}

\begin{proof}
Arrange the selected columns so that their leading rows increase from left to right, and restrict to those leading rows in the same increasing order. A selected column has zero entries in every row larger than its own leading row. The resulting square submatrix is therefore upper triangular. Each diagonal entry is the nonzero leading coefficient of a column, which equals one over $\F$. Its determinant is one, so the square submatrix is invertible. Any relation among the original selected columns restricts to a relation among these invertible columns and hence has every coefficient zero.
\end{proof}

The next result relates the row-order formulation to the uniquely restricted matching construction of Golumbic, Hirst, and Lewenstein~\cite[Section~1 and Theorem~2.1]{GolumbicHirstLewenstein2001}. The matching condition concerns the subgraph induced by the matched vertices; unmatched rows impose additional order constraints but cannot themselves lie on a directed cycle.

\begin{theorem}\label{thm:optimal-order}
For a finite matrix $\mathsf A$ over $\F$, with support graph $G(\mathsf A)$, one has
\begin{equation}\label{eq:optimal-order}
 \tau(\mathsf A)=\max_\prec\lambda_\prec(\mathsf A)
 =\nu_{\rm ur}(G(\mathsf A))\leq\rank\mathsf A.
\end{equation}
\end{theorem}

\begin{proof}
Fix a total row order and select one column for each distinct leading row. Choose its leading row as its pivot. The pivots are distinct by construction, and every arrow in \eqref{eq:pivot-arrows} points from a smaller row to a larger row. A directed cycle would give a strict chain beginning and ending at the same row, which is impossible. Thus $\lambda_\prec(\mathsf A)\leq\tau(\mathsf A)$.

Conversely, consider selected columns $\mathscr C$ with distinct pivots $\pi(c)$ whose constraint graph is acyclic. Every finite acyclic directed graph admits a topological order: there is a vertex with no incoming arrow, because otherwise repeatedly following incoming arrows would eventually repeat a vertex and produce a directed cycle; remove such a vertex and repeat. In the resulting order, every arrow points forward. Applied to the row constraints, this makes $\pi(c)$ larger than every other row in the support of $c$. Insert any remaining isolated rows anywhere in the order. Each selected pivot is then the leading row of its column, so $|\mathscr C|\leq\lambda_\prec(\mathsf A)$. Maximizing in both directions proves the first equality. Lemma~\ref{lem:triangular} gives the rank inequality.

For the matching equality, match each selected column $c$ to its pivot $\pi(c)$. Distinct columns and distinct pivots make this a matching in $G(\mathsf A)$. Every vertex with an incoming arrow in the row constraint graph is a pivot. Hence a directed cycle consists entirely of pivot rows. Write a simple such cycle as
\[
 \pi(c_1)\longrightarrow\pi(c_2)\longrightarrow\cdots
 \longrightarrow\pi(c_j)\longrightarrow\pi(c_1).
\]
By the definition of the arrows, $\pi(c_i)$ is incident with $c_{i+1}$, with indices taken cyclically. Each such incidence is outside the chosen matching, whereas $c_i$ is matched to $\pi(c_i)$. Alternating these two types of edges produces an alternating cycle in the support graph. Conversely, an alternating cycle for the chosen matching, traversed from each unmatched incidence to the next matched incidence, gives precisely a directed cycle among the pivot rows.

Finally, a matching is uniquely restricted if and only if there is no such alternating cycle. If an alternating cycle exists, exchange its matching edges with its other edges. The result is a different perfect matching on the same matched vertices. Conversely, if two different perfect matchings cover those vertices, their symmetric difference has degree two at every vertex it meets: each perfect matching contributes one incident edge, and common edges disappear. It is therefore a disjoint union of even cycles whose edges alternate between the two matchings. At least one cycle occurs because the matchings differ. This proves the equivalence directly. The pivot family is consequently acyclic exactly when its corresponding matching is uniquely restricted. Maximizing its size proves the second equality in \eqref{eq:optimal-order}.
\end{proof}

A proposed acyclic pivot family can be checked by a topological sort. Computing a maximum uniquely restricted matching is NP-hard for general bipartite graphs~\cite{GolumbicHirstLewenstein2001}; this fact does not establish the same complexity for the particular graphs $G(k,d)$. For a matrix whose rows are monomials of one fixed degree, the value $\lamstd(\mathsf A)$ obtained from the four prescribed orders is a computable lower bound for $\tau(\mathsf A)$ and need not attain it.

The finite optimization introduced in Section~\ref{subsec:prelim-graphs} also makes the order constraints explicit. With binary variables $y_{c,v}$ and $s_c$, and integer heights $1\leq H_v\leq n_{\rm row}(\mathsf A)$, impose
\[
 \sum_{v\in\Supp(c)}y_{c,v}=s_c,\qquad
 \sum_{c:v\in\Supp(c)}y_{c,v}\leq1,
\]
and, for every $v\in\Supp(c)$ and every $w\in\Supp(c)\setminus\{v\}$, impose
\[
 H_v\geq H_w+1-n_{\rm row}(\mathsf A)(1-y_{c,v}).
\]
The first equations select at most one pivot per column, and the second inequalities use each row at most once. If $y_{c,v}=1$, the height inequality makes the pivot strictly higher than every other support row of that column, precluding a directed cycle. If $y_{c,v}=0$, its right-hand side is at most one, so it imposes no restriction beyond the permitted height range. Conversely, a topological order for an acyclic pivot family supplies admissible heights. Maximizing $\sum_cs_c$ therefore gives exactly $\tau(\mathsf A)$. A matrix with no rows has no nonzero columns and has value zero without any height variables.

\subsection{Columns with acyclic incidence graph}\label{subsec:incidence-forests}

One can sometimes establish independence directly from the full incidence graph, without first choosing pivots. The forest condition below supplies an explicit pivot-selection procedure and shows that this construction is included in the preceding bound.

\begin{proposition}\label{lem:Berge-independent}
Let $\mathscr C$ be a finite family of columns, each with support size at least two. If their bipartite incidence graph is a forest, then the columns are linearly independent and admit acyclic pivots.
\end{proposition}

\begin{proof}
If there are no selected columns, the assertion is immediate. Otherwise, consider a component containing a column. It is a finite tree with at least one edge and therefore has a leaf. Every column vertex has degree at least two, so the leaf can be chosen to be a row vertex $v$. Let $c$ be its unique neighbouring column and choose $v$ as the pivot of $c$. Delete $c$ and then delete any isolated rows. The remaining incidence graph is a forest. Moreover, every remaining column retains all its original incident rows: a row incident with a remaining column is not isolated and is therefore not deleted. Thus the support-size hypothesis continues to hold, and the process can be repeated until all columns have been assigned pivots.

The pivots are distinct because a chosen row becomes isolated when its unique neighbouring column is removed. Place the unused rows first, followed by the pivots in reverse order of their selection. Fix a selected column $c$. Its own pivot occurs after every later-selected pivot. It cannot contain any earlier-selected pivot, because that row was a leaf adjacent only to the earlier-selected column while $c$ was still present. Every other support row of $c$ is therefore either unused or a later-selected pivot. In both cases it occurs before the pivot of $c$. Thus every selected column has its assigned pivot as its leading row in this order. The pivots are consequently acyclic, and Lemma~\ref{lem:triangular} proves independence.

Equivalently, if a nonempty sum of these columns vanished, the incidence graph restricted to the columns appearing in the sum would still be a forest with every column degree at least two. A row leaf would have coefficient one in that sum, contradicting its vanishing.
\end{proof}

For columns of support size two, this is the usual independence of graph-incidence columns belonging to a forest. Larger supports can also give independent columns when the incidence graph contains cycles, so the condition is sufficient rather than necessary.

\subsection{Congruence classes and simplicial boundary matrices}\label{subsec:simplicial-minors}

The congruence monomials in Section~\ref{subsec:prelim-simplicial} have more structure than arbitrary columns of a specified support size. Their binary digits force the column formula \eqref{eq:q-boundary}, and the index $\mathbf u$ separates the rows into disjoint components. Within each component, the matrix is an augmented boundary map for a full simplex. We compute its rank by exhibiting both an invertible minor and a spanning family of that same size.

\begin{proposition}\label{prop:q-components}
Assume \eqref{eq:q-degree}, and fix $\mathbf u\in\NN^k$ with $|\mathbf u|=\mathfrak n_{t,q}(d)+q$ and $s=|\vsupp(\mathbf u)|\geq q$. The columns in \eqref{eq:q-boundary}, indexed by the $q$-subsets of $\vsupp(\mathbf u)$, form the augmented simplicial boundary matrix
\[
 \mathsf d_{q-1}:\mathsf C_{q-1}(\Delta^{s-1};\F)
 \longrightarrow\mathsf C_{q-2}(\Delta^{s-1};\F).
\]
Its rank is $\binom{s-1}{q-1}$. Distinct component indices $\mathbf u$ have disjoint row sets. When $q=1$, the matrix is the augmentation map with $s$ columns and one row, and has rank one.
\end{proposition}

\begin{proof}
The basis vectors of $\mathsf C_{q-1}(\Delta^{s-1};\F)$ are the $q$-element subsets $I$ of the vertex set $\vsupp(\mathbf u)$. Over $\F$, their boundaries are the sums of their $(q-1)$-element subsets. Under the correspondence between a face $J$ and the monomial $f_{\mathbf u,J}$, this is exactly \eqref{eq:q-boundary}. For $q=1$, the boundary sends every singleton vertex to the basis element of the augmented chain group $\mathsf C_{-1}$, represented here by the empty subset. This verifies the identification, including its boundary case.

Choose a vertex $v_0\in\vsupp(\mathbf u)$ and retain the columns indexed by the $q$-subsets $I$ containing $v_0$. In the column indexed by $I$, the row $I\setminus\{v_0\}$ occurs with coefficient one. It occurs in no other retained column: any other $q$-subset containing $v_0$ and having this $(q-1)$-face would have to equal $I$. Every other face of $I$ contains $v_0$, so none is a selected row of the form $I'\setminus\{v_0\}$. The selected rows and columns therefore form an identity submatrix of size $\binom{s-1}{q-1}$.

To show that these columns span, consider a $q$-subset $I$ not containing $v_0$. Then
\[
 \mathsf d_{q-1}I
 =\sum_{v\in I}\mathsf d_{q-1}\bigl((I\setminus\{v\})\cup\{v_0\}\bigr).
\]
Indeed, deleting $v_0$ on the right gives the distinct faces $I\setminus\{v\}$, each once. Every other face has the form $(I\setminus\{v,w\})\cup\{v_0\}$ with $v\ne w$, and it occurs twice, once from each order of the two deleted vertices. These pairs cancel over $\F$. When $q=1$, there is only the empty face and both sides equal the augmentation generator, so the identity remains valid. Every column is thus in the span of the retained columns. Together with the identity minor, this proves the rank formula.

Finally, the exponent vector $e$ of a row monomial recovers its component index. The subset $J$ is exactly the set of positions at which $e_i\equiv2^t\pmod{2^{t+1}}$. At these positions,
\[
 u_i=\frac{e_i-2^t}{2^{t+1}}+1,
\]
and at every other position $u_i=e_i/2^{t+1}$. Thus a row belongs to at most one component $\mathbf u$, proving disjointness.
\end{proof}

The component description turns the total rank into a count of positive compositions. It also provides incidence forests inside every component, although their sizes need not equal the full component ranks.

\begin{theorem}\label{thm:simplicial-rank}
Assume \eqref{eq:q-degree}. The rank of all columns \eqref{eq:q-source} is
\begin{equation}\label{eq:B-tq}
 B_{t,q}(k,d)=\sum_{s=q}^{\min(k,\mathfrak n_{t,q}(d)+q)}
 \binom{s-1}{q-1}\binom{k}{s}\binom{\mathfrak n_{t,q}(d)+q-1}{s-1}.
\end{equation}
For $q\geq2$, these columns contain a family with acyclic incidence graph of size
\begin{equation}\label{eq:H-tq}
 F_{t,q}(k,d)=\sum_{s=q}^{\min(k,\mathfrak n_{t,q}(d)+q)}
 (s-q+1)\binom{k}{s}\binom{\mathfrak n_{t,q}(d)+q-1}{s-1}.
\end{equation}
Both counts are zero when $\mathfrak n_{t,q}(d)$ is not a nonnegative integer.
\end{theorem}

\begin{proof}
A component with support size $s$ is specified by choosing the support of $\mathbf u$ and then its positive entries. The support has $\binom{k}{s}$ choices. Its entries form a positive composition of $\mathfrak n_{t,q}(d)+q$ into $s$ parts, which has $\binom{\mathfrak n_{t,q}(d)+q-1}{s-1}$ choices: subtract one from every part and count weak compositions of the remaining total. Each component has rank $\binom{s-1}{q-1}$ by Proposition~\ref{prop:q-components}. Components have disjoint row sets, so a relation among their combined columns restricts independently to each component. Their column spans therefore form a direct sum, and their ranks add. Summing the product of the three factors over all possible $s$ proves \eqref{eq:B-tq}.

For the forest construction, fix a $(q-1)$-subset $I_0$ of $\vsupp(\mathbf u)$. Retain the $q$-subsets $I_0\cup\{v\}$ with $v\notin I_0$. There are $s-q+1$ such columns. Each contains the common row $I_0$. Its other rows are $(I_0\setminus\{w\})\cup\{v\}$ for $w\in I_0$. Each such row contains the unique vertex $v$ outside $I_0$, so it cannot occur in a retained column indexed by a different outside vertex. Thus the incidence graph consists of one common row joined to all selected columns, with $q-1$ additional row leaves attached to each column. It is connected and has no cycle, hence is a tree. For $q\geq2$, every column has support size at least two, so Proposition~\ref{lem:Berge-independent} also gives its independence. Taking these trees in all disjoint components and counting them as above proves \eqref{eq:H-tq}.
\end{proof}

For fixed $t$, a row in the $q$-family has exactly $q-1$ exponent residues equal to $2^t$ modulo $2^{t+1}$. Different values of $q$ therefore also have disjoint row sets, justifying the sum \mbox{$B_t(k,d)=\sum_{q=1}^kB_{t,q}(k,d)$} from Section~\ref{subsec:prelim-simplicial}. The next proposition identifies this total without evaluating the individual composition sums.

\begin{proposition}\label{prop:simplicial-collapse}
Let $d\geq1$ and $t\geq0$. Then
\begin{equation}\label{eq:B-collapse}
 B_t(k,d)=
 \begin{cases}
 \rho_k(d/2^t-1),&2^t\mid d,\\
 0,&2^t\nmid d.
 \end{cases}
\end{equation}
In particular,
\begin{equation}\label{eq:hierarchy-collapse}
 \max_{2^t\leq d}B_t(k,d)=\rho_k(d-1)
 \leq\lambda_{\rm lex}(M(k,d)).
\end{equation}
\end{proposition}

\begin{proof}
If $2^t\nmid d$, the degree equation \eqref{eq:q-degree} has no solution for any $q$, so every summand defining $B_t(k,d)$ is zero. Suppose that $2^t\mid d$. A monomial whose exponents are all divisible by $2^t$ is uniquely of the form $f^{2^t}$ with $f$ a monomial of degree $d/2^t-1$. Write the exponents of $f$ as twice nonnegative integers plus their parity digits, and let $I$ be the set of odd positions. If $I$ is nonempty, its size $q$ and the tuple of half-exponents identify precisely one of the congruence sources in \eqref{eq:q-source}. If $I$ is empty, $\Sq^1(f)=0$, so the corresponding column makes no rank contribution. Thus the union of the congruence families is exactly the set of nonzero $\Sq^{2^t}$ columns on domain monomials whose exponents are all divisible by $2^t$.

The Frobenius identity is
\[
 \Sq^{2^t}(f^{2^t})=(\Sq^1f)^{2^t};
\]
see \cite[\S1.3, Proposition~1.3.2, p.~7]{WalkerWoodI}. Taking the $2^t$th power is an injective $\F$-linear map on polynomials: it sends distinct monomials to distinct monomials and preserves sums in characteristic two. It therefore preserves the rank of the first-square column family. This gives $B_t(k,d)=\rho_k(d/2^t-1)$ and proves \eqref{eq:B-collapse}.

For $t\geq1$, every row in these columns has even exponents and positive degree. Choose any positive exponent and subtract one from it. The resulting monomial has exactly one odd exponent, so its $\Sq^1$ image is the original row monomial. Consequently the entire span of the congruence columns lies in $\im\Sq^1$ in degree $d$, and $B_t(k,d)\leq\rho_k(d-1)$. At $t=0$, the equality $B_0(k,d)=\rho_k(d-1)$ follows directly from \eqref{eq:B-collapse}. This proves the maximum identity.

For the comparison with leading rows, use the decomposition at $t=0$. Fix a component $\mathbf u$ and a column indexed by a $q$-subset $I$. Its terms are indexed by $I\setminus\{v\}$ with $v\in I$. Removing $v$ changes the exponent at $v$ from $2u_v-1$ to $2u_v$. If $v<w$ are two possible removed indices, the two terms agree before index $v$, and the term obtained by removing $v$ has the larger exponent at that first differing position. Hence exponent lexicographic order makes the term indexed by $I\setminus\{\min I\}$ the leading row.

Let $v_0=\min\vsupp(\mathbf u)$. A leading row cannot contain $v_0$: if a column contains $v_0$, that vertex is its minimum and is removed, while if the column does not contain $v_0$, none of its faces does. Conversely, every $(q-1)$-subset $J$ of $\vsupp(\mathbf u)\setminus\{v_0\}$ is a leading row, by taking the column $I=J\cup\{v_0\}$. Thus the component has exactly $\binom{s-1}{q-1}$ distinct lexicographic leading rows, equal to its rank. Distinct components and distinct $q$ have disjoint row sets, so these counts add to $\rho_k(d-1)$. Adding the higher layers cannot remove leading rows already present. The last inequality in \eqref{eq:hierarchy-collapse} follows.
\end{proof}

The comparisons above imply
\begin{equation}\label{eq:combined-reduced}
 \begin{aligned}
 \max\{D_k^\cup(d),\tau(M(k,d)),\max_tB_t(k,d)\}&=\tau(M(k,d)),\\
 \max\{D_k^\cup(d),\lamstd(M(k,d)),\max_tB_t(k,d)\}&=\lamstd(M(k,d)).
 \end{aligned}
\end{equation}
Indeed, every selected singleton output is a leading row in every order, and the simplicial total is already attained by the lexicographic first-layer count. The identities also hold at $d=0$, when all these rank bounds are zero. The simplicial decomposition nevertheless retains its separate role: it specifies independent columns and relations explicitly, even when its numerical contribution is already included in a leading-row bound.

\section{An explicit quotient by the first square}\label{sec:projection}
The exact first-square rank measures the dimension already accounted for by $\im\partial$, but it does not specify coordinates for the remaining columns. The projection defined in Section~\ref{subsec:prelim-projection} supplies such coordinates. We first prove the rank decomposition, then count the projected zero columns and classify the projected zero rows. The latter classification gives explicit functionals in the dual-spike module and leads to further independent families indexed by two-row tableaux.

The homotopy identity makes it possible to remove the first-square image before examining any higher layer. The following result identifies the resulting complement and explains how to form its columns directly.

\begin{theorem}\label{thm:Sq1-quotient}
Let $k\geq1$ and $d>0$. The map $p_d=h\partial$ defined in Section~\ref{subsec:prelim-projection} is a projection satisfying
\begin{equation}\label{eq:split-first-square}
 \im p_d=T_d(k),\qquad
 \ker p_d=\im(\partial:\mathcal P_k^{d-1}\longrightarrow\mathcal P_k^d),\qquad
 \dim T_d(k)=\rho_k(d).
\end{equation}
For $|e|=d$ and $i=i(e)$, its value is
\begin{equation}\label{eq:p-explicit}
 p_d(x^e)=
 \begin{cases}
 x^e,&e_i\text{ is odd},\\[2pt]
 \displaystyle\sum_{\substack{j>i\\e_j\text{ odd}}}
                x^{e-\eps_i+\eps_j},&e_i\text{ is even}.
 \end{cases}
\end{equation}
Every labelled column of $\overline M(k,d)$ satisfies
\begin{equation}\label{eq:projected-column}
 p_d\Sq^{2^t}(g)=h\Sq^{2^t+1}(g)\qquad(t\geq1),
\end{equation}
and the original and projected matrices have ranks related by
\begin{equation}\label{eq:rank-splitting}
 \rank M(k,d)=\rho_k(d-1)+\rank\overline M(k,d).
\end{equation}
\end{theorem}
\begin{proof}
The Cartan formula for the first square is
\[
 \partial(x^e)=\sum_{\substack{1\leq j\leq k\\e_j\text{ odd}}}x^{e+\eps_j}.
\]
Suppose first that $e_i$ is odd. The summand with $j=i$ has first positive exponent $e_i+1$, which is even, so $h(x^{e+\eps_i})=x^e$. Every other summand has the same first positive exponent $e_i$, which remains odd, and is therefore killed by $h$. Consequently $h\partial(x^e)=x^e$. If $e_i$ is even, then $e_i\geq2$ and the summand with $j=i$ is absent. Every remaining summand has first positive index $i$ and first positive exponent $e_i$. Applying $h$ subtracts one from that exponent and gives $x^{e-\eps_i+\eps_j}$. Since $e_i-1\geq1$, the first positive index stays equal to $i$. This proves \eqref{eq:p-explicit} in both cases, including the case of an empty sum.

Every nonzero term in \eqref{eq:p-explicit} has odd first positive exponent. Hence $\im p_d\subseteq T_d(k)$. Conversely, the first case of that formula shows that $p_d$ fixes each monomial of the basis $\mathcal T_d(k)$, so $T_d(k)\subseteq\im p_d$. These two inclusions prove the image assertion. Since $p_d$ is the identity on its image, $p_d^2=p_d$.

The relation $\partial^2=0$ gives $p_d\partial=h\partial^2=0$, proving $\im\partial\subseteq\ker p_d$. For the reverse inclusion, take $f\in\mathcal P_k^d$ with $p_d(f)=0$. Proposition~\ref{prop:contraction} applies because $d>0$ and gives
\[
 f=\partial h(f)+h\partial(f)=\partial h(f)+p_d(f)=\partial h(f).
\]
Thus $f$ belongs to the first-square image. We obtain the vector-space decomposition
\[
 \mathcal P_k^d=\im(\partial:\mathcal P_k^{d-1}\longrightarrow\mathcal P_k^d)\oplus T_d(k).
\]
Taking dimensions and using \eqref{eq:rho-recurrence}, we find
\[
 \dim T_d(k)=N_k(d)-\rho_k(d-1)=\rho_k(d),
\]
which completes the proof of \eqref{eq:split-first-square}.

For $t\geq1$, the integer $2^t$ is even, so the first-square identity in \cite[Proposition~1.3.4]{WalkerWoodI} gives
\[
 p_d\Sq^{2^t}(g)=h\partial\Sq^{2^t}(g)=h\Sq^{2^t+1}(g).
\]
This proves \eqref{eq:projected-column}.

Finally, restrict $p_d$ to the hit subspace $(\Aplus\Pk)_d$. Its kernel is exactly the first-square image, since that image is contained in the hit subspace and is the entire kernel of $p_d$ on $\mathcal P_k^d$. By \eqref{eq:outer-generators}, the hit subspace is spanned by all generating-square columns. The first-layer columns are killed by $p_d$, and the images of the remaining columns are precisely the columns of $\overline M(k,d)$. Thus the image of the restricted map is the column space of $\overline M(k,d)$. Rank-nullity for this restriction yields
\[
 \dim(\Aplus\Pk)_d=\rho_k(d-1)+\rank\overline M(k,d).
\]
The left side equals $\rank M(k,d)$, proving \eqref{eq:rank-splitting}.
\end{proof}

The complement depends on the ordering of the variables and need not be a $\GL(k,\F)$-submodule. For example, if $k=2$ and $\sigma$ exchanges $x_1$ and $x_2$, then
\[
 p_3\sigma(x_1^2x_2)=x_1x_2^2,
 \qquad \sigma p_3(x_1^2x_2)=x_1^2x_2.
\]
The first identity holds because $\sigma(x_1^2x_2)=x_1x_2^2$ has odd first exponent. For the second, \eqref{eq:p-explicit} first gives $p_3(x_1^2x_2)=x_1x_2^2$, and exchanging the variables then gives $x_1^2x_2$. Thus $p_3\sigma\ne\sigma p_3$.

For a fixed domain monomial $g$, every monomial in $\Sq^{2^t+1}(g)$ has the same set of variables with positive exponents as $g$; this support preservation is \cite[Proposition~1.1.11, p.~4]{WalkerWoodI}. Indeed, a zero domain exponent forces its allocation entry to be zero, whereas a positive domain exponent remains positive after addition. The first positive index is consequently the same in every term. On terms retained by $h$, subtraction of one at this fixed index is injective. Distinct retained exponent vectors therefore give distinct projected monomials, so no cancellation occurs in the direct expansion of \eqref{eq:projected-column}.

\subsection{The exact number of projected zero columns}
The preceding injectivity reduces projected nonvanishing to the existence of one suitable Cartan allocation. The residue sum $\mathfrak a_t(a)$ and the quotient and remainder $Q_t(n),r_t(n)$ are fixed in Section~\ref{subsec:prelim-spikes}, and the even-exponent count $\mathsf{Even}_k(n)$ is defined in Section~\ref{subsec:prelim-counts}.

\begin{proposition}\label{prop:projected-zero-columns}
Let $t\geq1$ and $a\in\NN^k$. Then
\[
 h\Sq^{2^t+1}(x^a)\ne0
 \quad\Longleftrightarrow\quad
 \mathfrak a_t(a)\geq2^t+1\text{ and at least one }a_j\text{ is odd}.
\]
For $n\geq0$, the number of zero columns obtained by projecting the $t$th-layer images of all degree-$n$ monomials is
\[
 \overline Z_t(k,n)
 =\mathsf{Even}_k(n)
 +\mathbf1_{r_t(n)\leq2^t}\binom{Q_t(n)+k-1}{k-1}
 \left(\binom{r_t(n)+k-1}{k-1}
 -\mathsf{Even}_k(r_t(n))\right).
\]
\end{proposition}
\begin{proof}
Suppose that the projected image is nonzero. Some allocation $b$ with $|b|=2^t+1$ and $b_j\preceq a_j$ therefore occurs in the Cartan expansion. Because $t\geq1$, we have $2^{t+1}>2^t+1$. No selected binary digit can consequently have weight $2^{t+1}$ or greater. The available digits are exactly those represented by the residues $a_j\bmod2^{t+1}$, and their total weight is $\mathfrak a_t(a)$. Hence $\mathfrak a_t(a)\geq|b|=2^t+1$. Moreover, $|b|$ is odd, so at least one selected digit has weight one. Its domain exponent $a_j$ has units digit one and is odd. This proves necessity.

Conversely, suppose that $\mathfrak a_t(a)\geq2^t+1$ and that at least one $a_j$ is odd. In particular $a\ne0$, so the first positive index $i=i(a)$ is defined. If $a_i$ is odd, select its units digit. If $a_i$ is even, select the units digit of any odd coordinate. Remove this chosen digit from the multiset of all available digits of weights at most $2^t$. The remaining total weight is at least $2^t$. Lemma~\ref{lem:binary-selection} supplies a selection of these remaining digits with total weight exactly $2^t$. Together with the previously chosen units digit, it gives an allocation $b$ of total $2^t+1$ satisfying $b_j\preceq a_j$ for every $j$.

If $a_i$ is odd, its units digit was included in $b_i$ and cannot be selected a second time. Therefore $b_i$ is odd and $a_i+b_i$ is even. If $a_i$ is even, its units digit is zero, so the binary inclusion $b_i\preceq a_i$ forces $b_i$ to be even. Again $a_i+b_i$ is even. In either case $x^{a+b}$ has first positive index $i$ with even exponent, and $h$ takes it to the nonzero monomial $x^{a+b-\eps_i}$. As observed after Theorem~\ref{thm:Sq1-quotient}, the retained terms have distinct images under $h$. This one surviving term cannot cancel, proving sufficiency.

We now count the zero columns. By the criterion, a domain monomial gives a zero column precisely when all its exponents are even or when $\mathfrak a_t(a)\leq2^t$. The first condition has $\mathsf{Even}_k(n)$ solutions. For the second condition, write uniquely
\[
 a_j=2^{t+1}\left\lfloor\frac{a_j}{2^{t+1}}\right\rfloor+\delta_j,
 \qquad0\leq\delta_j<2^{t+1}.
\]
Since $\sum_j a_j=n=2^{t+1}Q_t(n)+r_t(n)$, the residue sum $\sum_j\delta_j$ is congruent to $r_t(n)$ modulo $2^{t+1}$. The inequality $\sum_j\delta_j\leq2^t<2^{t+1}$ is consequently equivalent to
\[
 r_t(n)\leq2^t,\qquad
 \sum_j\delta_j=r_t(n),\qquad
 \sum_j\left\lfloor\frac{a_j}{2^{t+1}}\right\rfloor=Q_t(n).
\]
When $r_t(n)\leq2^t$, each nonnegative residue tuple of this sum automatically satisfies $\delta_j<2^{t+1}$. The residue tuples and quotient tuples are therefore counted independently by
\[
 \binom{r_t(n)+k-1}{k-1}
 \quad\text{and}\quad
 \binom{Q_t(n)+k-1}{k-1},
\]
respectively. Because $2^{t+1}$ is even, the domain tuple is all even exactly when each $\delta_j$ is even. There are $\mathsf{Even}_k(r_t(n))$ such residue tuples, with the same number of quotient tuples. Subtracting this overlap from the second count and adding the first count proves the stated formula.
\end{proof}

With the negative-degree convention from the preliminaries, \eqref{eq:projected-column-count} counts the nonzero labelled higher columns as $\overline{\mathcal W}_k(d)$. If $d$ is odd, then $d-2^t$ is odd for every $t\geq1$, so every domain tuple has an odd coordinate. Its residue sum $\mathfrak a_t(a)$ is also odd. Since $2^t$ is even, the conditions $\mathfrak a_t(a)\geq2^t$ and $\mathfrak a_t(a)\geq2^t+1$ are equivalent. Lemma~\ref{lem:annihilation} and Proposition~\ref{prop:projected-zero-columns} therefore show that the original and projected columns are nonzero for exactly the same higher-layer labels in odd output degree. In even output degree, the projection can kill additional columns.

\subsection{Dual functionals from projected zero rows}
Each coordinate on the complement $T_d(k)$ pulls back along $p_d$ to a functional that already vanishes on the first-square image. A zero row of the projected matrix is exactly the additional condition that makes this functional vanish on every higher-square image.

\begin{proposition}\label{prop:projected-dual}
Let $d>0$ and $x^e\in\mathcal T_d(k)$. If row $e$ of $\overline M(k,d)$ is zero, then the functional $\varphi_e$ defined in \eqref{eq:dual-phi} belongs to $K_k(d)$. The functionals obtained from all projected zero rows are linearly independent and include all coordinate dual spikes. In particular,
\begin{equation}\label{eq:projected-dual-bound}
 \dim\Qkd\geq z_k(d)\geq\mathcal S_k(d).
\end{equation}
\end{proposition}
\begin{proof}
Let $i=i(e)$, and consider the coordinate functional on $T_d(k)$ that takes the coefficient of $x^e$. We compute its pullback on a monomial $x^a\in\mathcal P_k^d$. If the first positive exponent of $a$ is odd, then $p_d(x^a)=x^a$, so the coefficient is one precisely when $a=e$. If the first positive exponent of $a$ is even, every term in \eqref{eq:p-explicit} has that same first positive index. A contribution to the $e$-coordinate therefore requires that index to equal $i$ and that
\[
 e=a-\eps_i+\eps_j
 \quad\text{for some }j>i\text{ with }a_j\text{ odd}.
\]
Equivalently, $a=e+\eps_i-\eps_j$ and $e_j$ is positive and even. This correspondence is reversible: such an $e_j$ gives $a_i=e_i+1$ positive and even, gives $a_j=e_j-1$ odd, and makes $x^e$ one of the indicated terms. Distinct indices $j$ give distinct domain exponent tuples. The pullback is thus exactly
\[
 (x^e)^*+
 \sum_{\substack{j>i\\e_j\in2\Z_{>0}}}(x^{e+\eps_i-\eps_j})^*
 =\varphi_e.
\]

Since $p_d$ kills the first-square image, this pullback vanishes on every first-layer column. If row $e$ of $\overline M(k,d)$ is zero, the $e$-coordinate of every higher projected column is zero, so the pullback also vanishes on every higher original column. These columns together span the hit space. Consequently $\varphi_e$ belongs to its annihilator $K_k(d)$.

The restriction of $\varphi_e$ to $T_d(k)$ is the single coordinate functional at $x^e$. Indeed, every extra dual monomial in its displayed expression has first positive exponent $e_i+1$, which is even, whereas the basis monomials of $T_d(k)$ have odd first positive exponent. If a linear combination of the zero-row functionals were zero, evaluating it on each corresponding basis monomial $x^e$ would recover its coefficient. All coefficients must therefore vanish. This proves independence.

For a positive-degree spike, every positive exponent is odd, so its first positive exponent is odd and the monomial belongs to $\mathcal T_d(k)$. There are no positive even later exponents, and hence $\varphi_e=(x^e)^*$. Lemma~\ref{lem:spike-not-term} says that this coordinate functional vanishes on the entire hit space. It follows from the pullback description that the corresponding projected row is zero. Thus all coordinate dual spikes occur among the independent functionals. Their number is $\mathcal S_k(d)$, while the complete family has $z_k(d)$ members. Proposition~\ref{prop:dual-certificate} now gives \eqref{eq:projected-dual-bound}.
\end{proof}

\subsection{Classification of the projected zero rows}
To decide whether a projected row occurs, one can first reverse the homotopy and then ask for a Cartan allocation into a single specified monomial. Binary borrowing supplies such an allocation unless the exponents have one of the two forms in the next theorem.

\begin{theorem}\label{thm:projected-zero-classification}
Let $k\geq1$, $d>0$, and $x^e\in\mathcal T_d(k)$, and put $i=i(e)$. Row $e$ of $\overline M(k,d)$ is zero if and only if $x^e$ is a spike, or $e_i=1$, exactly one index $j>i$ has $e_j=2$, and every exponent outside positions $i,j$ is a spike exponent. Consequently,
\begin{equation}\label{eq:closed-projected-zero-count}
 z_k(d)=\mathcal S_k(d)+
          \sum_{s=0}^{k-2}(s+1)\mathcal S_s(d-3),
\end{equation}
where the sum is empty for $k=1$.
\end{theorem}
\begin{proof}
The unique monomial that $h$ takes to $x^e$ is $x^{e+\eps_i}$. In fact, $h$ decreases a positive even exponent, which is at least two, and therefore preserves its first positive position. To obtain $e$, that position must be $i$ and its exponent before applying $h$ must be $e_i+1$. By \eqref{eq:projected-column}, the projected row is nonzero precisely when there are $t\geq1$ and exponent tuples $a,b\in\NN^k$ satisfying
\begin{equation}\label{eq:projected-row-witness}
 a+b=e+\eps_i,\qquad |b|=2^t+1,\qquad
 b_j\preceq a_j\quad(1\leq j\leq k).
\end{equation}
Such tuples exhibit $x^{e+\eps_i}$ as a nonzero term of $\Sq^{2^t+1}(x^a)$. Since retained terms have distinct images under $h$, this occurrence is sufficient to make row $e$ nonzero.

Assume that row $e$ is zero. Its first positive exponent $e_i$ is odd by the definition of $\mathcal T_d(k)$. If it is not a spike exponent, there is a binary position $t\geq1$ below its highest nonzero digit whose digit is zero. Set
\[
 a=e-2^t\eps_i,\qquad b=(2^t+1)\eps_i.
\]
The higher nonzero digit guarantees $e_i>2^t$, so $a$ is nonnegative. When $2^t$ is subtracted, borrowing from a higher nonzero position changes the digit at position $t$ from zero to one. Digits below position $t$ are unchanged, in particular the units digit remains one. Thus $a_i$ has both digits required by $2^t+1$, and $2^t+1\preceq a_i$. All other allocation entries are zero, so \eqref{eq:projected-row-witness} holds. This contradicts the assumption of a zero row. Therefore
\[
 e_i=2^u-1\qquad\text{for some }u\geq1.
\]

Now take a later positive exponent $e_j$, with $j>i$. If its binary expansion has a zero digit at any position $t\geq1$ below its highest nonzero digit, set
\[
 a=e-2^t\eps_j,
 \qquad b=\eps_i+2^t\eps_j.
\]
The odd number $a_i=e_i$ contains the required units digit, and the same borrowing argument gives a one at position $t$ of $a_j$. Hence $b_i\preceq a_i$, $b_j\preceq a_j$, and all other inclusions are automatic. Again \eqref{eq:projected-row-witness} holds. For a zero row, every positive later exponent must consequently have ones at all binary positions from one through its highest position. Its units digit is either one or zero, so it has the form $2^v-1$ or $2^v-2$. A positive exponent of the second form has $v\geq2$.

If every later exponent is a spike exponent, then $x^e$ is itself a spike, which is one of the stated possibilities. Otherwise choose $j>i$ with $e_j=2^v-2$. Suppose that $u\geq2$. Taking
\[
 b=2\eps_i+\eps_j,
 \qquad a=e+\eps_i-b
\]
gives $a_i=2^u-2$ and $a_j=2^v-3$. The first has binary digit one at position one, and the second is odd. Thus $2\preceq a_i$ and $1\preceq a_j$. Since $|b|=3=2^1+1$, these tuples satisfy \eqref{eq:projected-row-witness} with $t=1$. We conclude that $u=1$, so $e_i=1$.

Suppose next that the chosen even exponent has $v\geq3$. Take
\[
 b=3\eps_j,
 \qquad a=e+\eps_i-b.
\]
Here $a_j=2^v-5$ is nonnegative and is congruent to three modulo four, so its two lowest binary digits are one. Hence $3\preceq a_j$, and the other allocation entries are zero. This again gives \eqref{eq:projected-row-witness} with $t=1$. Consequently every positive even exponent after position $i$ must equal two.

There cannot be two such exponents. Indeed, if $e_j=e_{j'}=2$ for distinct $j,j'>i$, choose
\[
 b=\eps_i+\eps_j+\eps_{j'},
 \qquad a=e+\eps_i-b.
\]
The three domain exponents at these positions are all one, and the three corresponding allocation entries are all one. The binary inclusions hold and $|b|=3$, yielding the same contradiction. Thus a non-spike zero row has $e_i=1$ and exactly one later exponent equal to two; all other exponents are spikes. Every domain tuple used above is nonnegative. Moreover, the binary inclusions imply $|a|\geq|b|=2^t+1$, so the constructions always respect instability and use an actual layer of $M(k,d)$.

For the converse, spike rows are zero by Proposition~\ref{prop:projected-dual}. Consider the other stated form, with $e_i=1$ and $e_j=2$ at its unique later non-spike position. The tuple $e+\eps_i$ has exponent two at positions $i$ and $j$ and spike exponents everywhere else. Lemma~\ref{lem:spike-not-term}, applied in one variable, forbids a positive Cartan increment into a spike exponent. An increment into exponent two can only be zero or one: increment two would leave domain exponent zero and have coefficient $\binom02=0$. Thus every surviving allocation into $e+\eps_i$ has total at most two. Equation~\eqref{eq:projected-row-witness} requires total $2^t+1\geq3$, which is impossible. The row is therefore zero.

It remains to count the non-spike zero rows. Fix their first positive index $i$. There are $k-i$ possible positions for the unique exponent two. Once that position is chosen, the other $k-i-1$ later positions contain arbitrary spike exponents of total $d-3$. All positions before $i$ are zero, while the prescribed exponents one and two contribute three to the degree. Hence this index contributes
\[
 (k-i)\mathcal S_{k-i-1}(d-3)
\]
rows. Different first positive indices or different positions of the exponent two give different rows, and none of them is a spike. Adding the $\mathcal S_k(d)$ spike rows and substituting $s=k-i-1$ gives \eqref{eq:closed-projected-zero-count}. The negative-degree and empty-sum conventions include $d<3$ and $k=1$.
\end{proof}

The classification also determines where these functionals lie inside the full Steenrod kernel. The exceptional exponents one and two produce precisely the classical degree-three dual-spike combination.

\begin{corollary}\label{cor:projected-spike-module}
Let $k\geq1$ and $d>0$. The span of the functionals $\varphi_e$ indexed by the zero rows of $\overline M(k,d)$ is a $z_k(d)$-dimensional subspace of $J_k(d)$.
\end{corollary}
\begin{proof}
Proposition~\ref{prop:projected-dual} gives independence, so it suffices to prove membership in $J_k(d)$. For a spike row, $\varphi_e=(x^e)^*$ is one of the coordinate generators of that module. For a non-spike zero row, Theorem~\ref{thm:projected-zero-classification} supplies indices $i<j$ with $e_i=1$, $e_j=2$, and spike exponents at all other positions. The defining sum for $\varphi_e$ then contains exactly one additional term. Consequently
\[
 \begin{aligned}
 \varphi_e
 &=\bigl(v_i^{(1)}v_j^{(2)}+v_i^{(2)}v_j^{(1)}\bigr)
       \prod_{j'\notin\{i,j\}}v_{j'}^{(e_{j'})}\\
 &=\bigl((v_i+v_j)^{(3)}+v_i^{(3)}+v_j^{(3)}\bigr)
       \prod_{j'\notin\{i,j\}}v_{j'}^{(e_{j'})}.
 \end{aligned}
\]
The second equality follows by expanding the third divided power of a sum, as in \cite[Example~9.4.6, p.~175]{WalkerWoodI}. The terms with $v_i^{(3)}$ and $v_j^{(3)}$ are coordinate dual spikes. The remaining term is the image of the coordinate dual spike with factor $v_i^{(3)}$ under the invertible substitution $v_i\mapsto v_i+v_j$, fixing the other basis vectors. Since all remaining factors avoid positions $i,j$, that substitution fixes them. The definition of $J_k(d)$ therefore places all three terms, and hence $\varphi_e$, in $J_k(d)$.
\end{proof}

Thus projected coordinates can produce more independent functionals than the coordinate spikes alone, but every functional obtained from a projected zero row lies in the classical dual-spike module. Their classes in $K_k(d)/J_k(d)$ are consequently zero.

\subsection{The span obtained from all variable orders}
The complement used by the projection changes when the variables are reordered. To combine the resulting functionals without counting their dependencies repeatedly, we use the space $\mathcal L_k^{\mathrm{ord}}(d)$ from Section~\ref{subsec:prelim-projection} and the restricted spike count $\mathcal S_s^{\ne1}(n)$ from Section~\ref{subsec:prelim-counts}. The relevant dependencies reduce to the elementary coefficient-sum-zero subspace of a permutation space.

\begin{proposition}\label{prop:all-variable-orders}
For $k\geq1$ and $d>0$, the space $\mathcal L_k^{\mathrm{ord}}(d)$ is contained in $J_k(d)$ and has dimension
\[
 \dim\mathcal L_k^{\mathrm{ord}}(d)
 =\mathcal S_k(d)+\sum_{s=2}^k(s-1)\binom{k}{s}
                      \mathcal S_{k-s}^{\ne1}(d-s-1).
\]
Consequently the right side is a lower bound for $\dim\Qkd$ and is at least $z_k(d)$.
\end{proposition}
\begin{proof}
Every ordering contributes all coordinate dual spikes, whose distinct monomial supports make them an independent family of size $\mathcal S_k(d)$. By Theorem~\ref{thm:projected-zero-classification}, every additional functional contributed by any ordering has the form
\[
 \bigl(v_i^{(1)}v_j^{(2)}+v_i^{(2)}v_j^{(1)}\bigr)
       \prod_{j'\notin\{i,j\}}v_{j'}^{(e_{j'})},
\]
with spike exponents outside the selected pair. Conversely, any functional of this form appears for an ordering that places variable $i$ first and variable $j$ later: the first monomial in the displayed expression then has precisely the exceptional zero-row form in the classification. Corollary~\ref{cor:projected-spike-module} shows that every such functional lies in $J_k(d)$. Hence their span $\mathcal L_k^{\mathrm{ord}}(d)$ is contained in that module.

To calculate the dimension, fix a subset $\Omega\subseteq\{1,\ldots,k\}$ of size $s\geq2$ and fix exponents $e_{j'}\in\{0,3,7,15,\ldots\}$ outside $\Omega$ such that their sum is $d-s-1$. For each $j\in\Omega$, consider the dual monomial
\[
 w_j=v_j^{(2)}\prod_{i\in\Omega\setminus\{j\}}v_i^{(1)}
             \prod_{j'\notin\Omega}v_{j'}^{(e_{j'})}.
\]
These $s$ monomials are distinct. The pair functionals with the prescribed outside exponents and with positions of exponents one or two exactly equal to $\Omega$ are the vectors $w_i+w_j$, for distinct $i,j\in\Omega$. Their span is the coefficient-sum-zero subspace
\[
 \left\{\sum_{j\in\Omega}\gamma_jw_j:
               \gamma_j\in\F,\quad\sum_{j\in\Omega}\gamma_j=0\right\}.
\]
Indeed, every pair vector has coefficient sum zero. Conversely, choose one $j_0\in\Omega$. If a vector has coefficient sum zero, then $\gamma_{j_0}=\sum_{j\ne j_0}\gamma_j$ in $\F$, and the vector equals
\[
 \sum_{j\in\Omega\setminus\{j_0\}}\gamma_j(w_j+w_{j_0}).
\]
The $s-1$ vectors in this expression are independent: the coefficient of $w_j$ for $j\ne j_0$ is the coefficient of its unique corresponding vector. Thus the prescribed data contribute dimension $s-1$.

These contributions form a direct sum. In any monomial $w_j$, the set of positions with exponent one or two is exactly $\Omega$, and the exponents outside it recover the prescribed tuple. Hence two different choices of these data have disjoint monomial supports. They are also disjoint from the coordinate dual spikes, since each $w_j$ has an exponent equal to two. For each $s$, there are $\binom{k}{s}$ choices of $\Omega$ and $\mathcal S_{k-s}^{\ne1}(d-s-1)$ choices of the outside tuple. Summing their dimensions and adding the coordinate spike dimension proves the formula.

The inclusion $\mathcal L_k^{\mathrm{ord}}(d)\subseteq K_k(d)$ and Proposition~\ref{prop:dual-certificate} give the lower cohit bound. The fixed ordering used to define $z_k(d)$ is one of the orderings included in $\mathcal L_k^{\mathrm{ord}}(d)$, and its zero-row functionals are independent. Therefore $\dim\mathcal L_k^{\mathrm{ord}}(d)\geq z_k(d)$.
\end{proof}

\subsection{Independent families from disjoint pairs of variables}
The degree-three pair functional can be multiplied across disjoint variable pairs. The definitions in Section~\ref{subsec:prelim-tableaux} select a family indexed by standard two-row tableaux. The next proof verifies its membership, independence, and cardinality directly, so no representation-theoretic basis theorem is needed to obtain the bound.

\begin{theorem}\label{thm:disjoint-pair-bound}
For $k\geq1$ and $d>0$, the coordinate dual spikes together with the vectors $\xi_{\Omega,e,\mathsf B}$ specified in Section~\ref{subsec:prelim-tableaux} form a linearly independent family of $\beta_k(d)$ elements in $J_k(d)$, where $\beta_k(d)$ is defined by \eqref{eq:paired-bound}. Consequently,
\[
 \dim\Qkd\geq\dim J_k(d)\geq\beta_k(d)
 \geq\dim\mathcal L_k^{\mathrm{ord}}(d)\geq z_k(d).
\]
\end{theorem}
\begin{proof}
Fix integers $s,p$ with $1\leq p\leq\lfloor s/2\rfloor$, a set $\Omega=\{o_1<\cdots<o_s\}\subseteq\{1,\ldots,k\}$, and outside exponents $e_{j'}\in\{0,3,7,15,\ldots\}$ of sum $d-s-p$. Let $\mathsf B=\{\mathfrak b_1<\cdots<\mathfrak b_p\}\subseteq\{1,\ldots,s\}$ satisfy $\mathfrak b_i\geq2i$, and write its complement as $\mathsf T=\{\mathfrak t_1<\cdots<\mathfrak t_{s-p}\}$. Among positions $1,\ldots,\mathfrak b_i-1$, exactly $i-1$ lie in $\mathsf B$, so exactly $\mathfrak b_i-i$ lie in $\mathsf T$. Since $\mathfrak b_i-i\geq i$, at least $i$ entries of $\mathsf T$ precede $\mathfrak b_i$, giving $\mathfrak t_i<\mathfrak b_i$. Thus these rows form a standard tableau of shape $(s-p,p)$.

The vector prescribed in the preliminaries is
\[
 \xi_{\Omega,e,\mathsf B}=
 \prod_{i=1}^p
 \left(v_{o_{\mathfrak t_i}}^{(1)}v_{o_{\mathfrak b_i}}^{(2)}
       +v_{o_{\mathfrak t_i}}^{(2)}v_{o_{\mathfrak b_i}}^{(1)}\right)
 \prod_{i=p+1}^{s-p}v_{o_{\mathfrak t_i}}^{(1)}
 \prod_{j'\notin\Omega}v_{j'}^{(e_{j'})}.
\]
No variable occurs in two different factors, so the divided-power multiplication produces no additional binomial coefficients. The pair factors contribute degree $3p$, the unpaired variables contribute $s-2p$, and the outside exponents contribute $d-s-p$. Their sum is $d$.

For membership in $J_k(d)$, replace each pair factor by
\[
 (v_{o_{\mathfrak t_i}}+v_{o_{\mathfrak b_i}})^{(3)}
       +v_{o_{\mathfrak t_i}}^{(3)}+v_{o_{\mathfrak b_i}}^{(3)}.
\]
After expanding the product, each summand has one third divided power on each selected pair, with its underlying linear form equal to $v_{o_{\mathfrak t_i}}+v_{o_{\mathfrak b_i}}$, $v_{o_{\mathfrak t_i}}$, or $v_{o_{\mathfrak b_i}}$. Each of these nonzero linear forms can be extended to a basis of the two-dimensional coordinate space belonging to its pair. Because the pairs are disjoint, these changes of basis combine with the identity on all remaining variables to form an invertible substitution on the full $k$-dimensional space. The summand is therefore the image of a coordinate dual spike with exponent three on one coordinate in each pair, exponent one on every unpaired coordinate of $\Omega$, and the prescribed spike exponents outside $\Omega$. Thus every summand lies in $J_k(d)$, and so does their sum $\xi_{\Omega,e,\mathsf B}$. This explicit argument is consistent with the subring property of the dual-spike module in \cite[Proposition~9.4.8]{WalkerWoodI}.

To prove independence for fixed $\Omega,e,p$, identify a monomial in these vectors with the $p$-element subset $\mathsf D\subseteq\{1,\ldots,s\}$ recording the positions in $\Omega$ at which its exponent is two. All other positions in $\Omega$ have exponent one. Order such subsets by their sums of indices, using any fixed lexicographic order only when the sums are equal. Choosing exponent two at $o_{\mathfrak b_i}$ in every pair produces the monomial indexed by $\mathsf B$, with coefficient one. Every different choice changes at least one exponent-two position from $\mathfrak b_i$ to $\mathfrak t_i$. Since $\mathfrak t_i<\mathfrak b_i$ for every pair and the pairs are disjoint, the sum of the selected indices strictly decreases. The monomial indexed by $\mathsf B$ is therefore the unique largest monomial in its vector.

Different permitted subsets $\mathsf B$ have different leading monomials. In a nonzero linear combination of their vectors, choose the largest leading monomial among those whose coefficients are one. It occurs in its own vector with coefficient one. Every other selected vector has a smaller leading monomial, and all terms of that vector are at most its leading monomial. Hence none can contain the chosen largest monomial. Its coefficient in the combination is one, proving that the combination is nonzero. This establishes independence for the fixed data.

To pass to the full family, observe that every monomial in these supports recovers its data: $\Omega$ is the set of positions with exponent one or two, $e$ consists of the exponents outside $\Omega$, and $p$ is the number of exponents equal to two. Thus distinct triples $(\Omega,e,p)$ have disjoint monomial supports. All these supports are disjoint from the coordinate dual spikes, because $p\geq1$ forces an exponent equal to two. Independence for each fixed triple therefore proves independence after taking the union over all triples and adding all coordinate dual spikes.

It remains to count the permitted bottom rows $\mathsf B$. Encode a $p$-element subset of $\{1,\ldots,s\}$ by a word of length $s$, using the symbol $\mathsf B$ at the selected positions and $\mathsf T$ elsewhere. The inequalities $\mathfrak b_i\geq2i$ say exactly that every initial segment contains at least as many $\mathsf T$ symbols as $\mathsf B$ symbols. There are $\binom{s}{p}$ words before imposing this restriction.

For a violating word, take its shortest initial segment with more $\mathsf B$ symbols than $\mathsf T$ symbols. This segment has exactly one more $\mathsf B$ than $\mathsf T$. Interchange the two symbols within that initial segment. This reduces the total number of $\mathsf B$ symbols by one and produces a word with $p-1$ such symbols. Before the selected endpoint, the original excess of $\mathsf B$ over $\mathsf T$ was nonpositive; after interchange, the excess of $\mathsf T$ over $\mathsf B$ is nonpositive. At the endpoint it becomes one. Thus the selected segment is the first segment with one more $\mathsf T$ than $\mathsf B$ in the new word.

Conversely, a word with $p-1$ symbols $\mathsf B$ has $s-p+1$ symbols $\mathsf T$. Since $s\geq2p$, its final excess of $\mathsf T$ over $\mathsf B$ is $s-2p+2\geq2$. Its excess starts at zero and changes by one at each position, so it has a first initial segment with excess one. Interchanging the symbols on that segment creates a violating word with $p$ symbols $\mathsf B$, and the same first-crossing argument shows that this is the inverse construction. Hence the violating words are in bijection with all words having $p-1$ symbols $\mathsf B$. Their number is $\binom{s}{p-1}$, and the number of permitted bottom rows is
\[
 \binom{s}{p}-\binom{s}{p-1}.
\]

There are $\binom{k}{s}$ choices of $\Omega$ and $\mathcal S_{k-s}^{\ne1}(d-s-p)$ choices of the outside exponents. Multiplying these counts by the number of permitted bottom rows, summing over $s,p$, and adding the $\mathcal S_k(d)$ coordinate dual spikes gives exactly \eqref{eq:paired-bound}. The family therefore has cardinality $\beta_k(d)$. Since it lies in $J_k(d)\subseteq K_k(d)$, the first two lower bounds follow from duality. For $p=1$, the factor $\binom{s}{1}-\binom{s}{0}$ equals $s-1$, so the coordinate-spike term together with all $p=1$ summands is exactly the dimension in Proposition~\ref{prop:all-variable-orders}. Every further summand is nonnegative. This proves the remaining comparisons.
\end{proof}

\begin{remark}\label{rem:two-row-polytabloids}
For fixed $\Omega,e,p$, the identification by exponent-two positions takes a row tabloid of shape $(s-p,p)$ to the subset occupied by its second row. The column stabilizer independently exchanges $\mathfrak t_i$ and $\mathfrak b_i$ in each two-entry column. In characteristic two the signs in the column antisymmetrizer are all one, so its polytabloid expands as the pair-factor product defining $\xi_{\Omega,e,\mathsf B}$. This is exactly the construction in \cite[Chapter~4, Definitions~4.1 and~4.3, p.~13]{James1978}; standard tableaux and their basis theorem are given in \cite[Chapter~8, Definition~8.1 and Theorem~8.4, p.~29]{James1978}. The argument above proves the required independence directly and counts the standard tableaux by a reflection bijection. It establishes an independent family in $J_k(d)$ and does not assert that this family spans either $J_k(d)$ or $K_k(d)$.
\end{remark}

\section{Two-sided bounds and comparison with known dimensions}\label{sec:main}
We combine the unprojected and projected minors with the relation and dual bounds, and then evaluate the resulting inequalities in small degrees and at $(k,d)=(5,29)$. These calculations compare the effects of the row order, the first-square projection, and the available lower-bound constructions. Each estimate is applied separately, so intersections between independently constructed subspaces need not be assumed trivial.

The rank bounds $\mathsf{LB}_{\mathrm{eval}}$, $\mathsf{LB}_{\mathrm{opt}}$, and $\mathsf{UB}$ are those defined in Section~\ref{subsec:prelim-combined}. Their separate inputs are combined by a maximum for lower rank bounds and a minimum for upper rank bounds; no additivity between unrelated constructions is required.

\begin{theorem}\label{thm:main-rank}
Let $k\geq1$ and $d\geq0$. For $d>0$, let $\mathcal U\subseteq K_k(d)$ and $\mathscr R$ satisfy the hypotheses just specified. Then
\begin{equation}\label{eq:main-rank}
 \mathsf{LB}_{\mathrm{eval}}(k,d)\leq\mathsf{LB}_{\mathrm{opt}}(k,d)
 \leq\rank M(k,d)\leq\mathsf{UB}(k,d).
\end{equation}
Consequently,
\begin{equation}\label{eq:main-dimension}
 N_k(d)-\mathsf{UB}(k,d)\leq\dim\Qkd
 \leq N_k(d)-\mathsf{LB}_{\mathrm{opt}}(k,d)
 \leq N_k(d)-\mathsf{LB}_{\mathrm{eval}}(k,d).
\end{equation}
In degree zero, $\dim(Q\mathcal P_k)_0=1$.
\end{theorem}
\begin{proof}
Assume first that $d>0$. For either matrix $\mathsf A=M(k,d)$ or $\overline M(k,d)$, the four specified orders form a subset of all total row orders. Theorem~\ref{thm:optimal-order} therefore gives $\lamstd(\mathsf A)\leq\tau(\mathsf A)\leq\rank\mathsf A$. Applying this to the original matrix gives the first lower-rank input. Applying it to the projected matrix and then adding $\rho_k(d-1)$ gives the second input, because \eqref{eq:rank-splitting} identifies $\rho_k(d-1)+\rank\overline M(k,d)$ with $\rank M(k,d)$. Taking the respective maxima proves both inequalities on the left of \eqref{eq:main-rank}.

For the first upper-rank input, Proposition~\ref{prop:dual-certificate} gives $\dim\Qkd\geq\ell_k(d)$, and Theorem~\ref{thm:disjoint-pair-bound} gives $\dim\Qkd\geq\beta_k(d)$. Thus $\dim\Qkd\geq\max\{\ell_k(d),\beta_k(d)\}$, and \eqref{eq:dimension-rank} gives the first entry in \eqref{eq:upper-combined}. We take a maximum rather than adding these dimensions because the chosen space $\mathcal U$ may intersect the explicitly constructed dual family.

The second upper-rank input is Theorem~\ref{thm:Adem-upper}. Its correction is a dimension increment modulo $V_1(k,d)$, so it has already accounted for intersections among relation spaces. For the third input, deleting the zero rows and zero columns of $\overline M(k,d)$ leaves a matrix with $\overline N_k(d)$ rows and $\overline{\mathcal W}_k(d)$ columns and the same rank. The dimension of its image cannot exceed either of these two numbers. Hence
\[
 \rank\overline M(k,d)\leq\min\{\overline N_k(d),\overline{\mathcal W}_k(d)\}.
\]
Adding $\rho_k(d-1)$ and using \eqref{eq:rank-splitting} proves the third upper-rank input. Each input separately bounds the same rank, so their minimum does as well. This proves \eqref{eq:main-rank}.

Subtracting the rank inequalities from $N_k(d)$ reverses their directions. Substitution of \eqref{eq:dimension-rank} gives \eqref{eq:main-dimension}. If $d=0$, there is one row, indexed by the constant monomial, and no columns, since no power $2^t$ is at most zero. Thus $M(k,0)$ has rank zero, the three defined rank bounds are zero, and the degree-zero cohit quotient is one-dimensional.
\end{proof}

The choice $\mathcal U=0$ and $\mathscr R=\varnothing$ already incorporates the explicit lower cohit bound $\beta_k(d)$ and requires no external basis information. The projected-row bound is $\dim\Qkd\geq z_k(d)$, because $N_k(d)-\rho_k(d-1)=\dim T_d(k)$ and $\dim T_d(k)-\overline N_k(d)=z_k(d)$. Independent cohit classes or explicit dual kernels from the literature may strengthen the chosen $\ell_k(d)$, and additional evaluated relations may strengthen $\delta_{\mathscr R}^{(1)}$. We now evaluate the bounds from the selected support families in small degrees and at $(k,d)=(5,29)$, incorporating classical spike-insertion information separately. The entries in the tables below are obtained by enumerating the surviving allocations in \eqref{eq:cartan} and \eqref{eq:projected-column} and counting the corresponding support rows in the four prescribed orders.

\subsection{Order dependence in two variables}
In degree $5$, identify a row with the exponent $e_1$ of $x_1$. The nonzero supports are
\[
 \{1,2\},\ \{3,4\},\ \{0\},\ \{1\},\ \{4\},\ \{5\}.
\]
Thus the exact singleton-target count is four, whereas the selected count $D_2^{\cup}(5)$ is two. Exponent lexicographic and degree-reverse-lexicographic orders give five leading rows. Both specified weight orders give six, because the weights $(1,2)$ at rows $2,3$ are larger than $(1,0,1)$ at rows $1,4$ in both conventions. Therefore $\rank M(2,5)=\tau(M(2,5))=6$ and $\dim(Q\mathcal P_2)_5=0$.

In degree $7$, the five supports form a path whose consecutive vertices are $1,2,4,3,5,6$. The four predetermined orders each give four leading rows. The order $0\prec1\prec2\prec4\prec3\prec5\prec6\prec7$ gives five. Thus $\tau(M(2,7))=\rank M(2,7)=5$ and $\dim(Q\mathcal P_2)_7=3$. These examples show that the four prescribed orders need not attain the maximum triangular-minor size.

\subsection{Three variables in degree seven}\label{subsec:k3d7}
Here $N_3(7)=36$, $\mathcal S_3(7)=6$, $\mathcal W_3(7)=30$, and $\rho_3(6)=15$. The nonzero first-layer count is $18$, so $\delta_1(3,7)=3$ and $\mathcal W_3^{(1)}(7)=27$. The four original leading-row counts are $20,20,23,20$. After the first-square projection, the $\Sq^2$ layer alone has eleven distinct lexicographic leading rows, as the following list of eleven columns shows. An exponent tuple in the last column denotes a monomial of degree seven.
\begin{center}
\begin{tabular}{ccl}
\toprule
Domain of $\Sq^2$&Leading row&Support of $p_7\Sq^2$\\\midrule
$(0,2,3)$&$(0,3,4)$&$(0,1,6)+(0,3,4)$\\
$(0,3,2)$&$(0,5,2)$&$(0,3,4)+(0,5,2)$\\
$(1,1,3)$&$(1,2,4)$&$(1,1,5)+(1,2,4)$\\
$(1,2,2)$&$(1,4,2)$&$(1,2,4)+(1,4,2)$\\
$(1,3,1)$&$(1,5,1)$&$(1,4,2)+(1,5,1)$\\
$(2,0,3)$&$(3,0,4)$&$(1,0,6)+(3,0,4)$\\
$(2,1,2)$&$(3,2,2)$&$(1,2,4)+(3,2,2)$\\
$(2,3,0)$&$(3,4,0)$&$(1,6,0)+(3,4,0)$\\
$(3,0,2)$&$(5,0,2)$&$(3,0,4)+(5,0,2)$\\
$(3,1,1)$&$(5,1,1)$&$(3,2,2)+(5,1,1)$\\
$(3,2,0)$&$(5,2,0)$&$(3,4,0)+(5,2,0)$\\\bottomrule
\end{tabular}
\end{center}
The table follows by expanding $h\Sq^3$ on the stated domain monomials. For example, the permitted allocations of total three into $(0,2,3)$ are $(0,0,3)$ and $(0,2,1)$. Thus $\Sq^3(x_2^2x_3^3)=x_2^2x_3^6+x_2^4x_3^4$, and applying $h$ lowers the first positive exponent by one in both terms, giving $x_2x_3^6+x_2^3x_3^4$. The eleven leading tuples in the middle column are distinct. Lemma~\ref{lem:triangular} therefore gives eleven independent projected columns, and \eqref{eq:rank-splitting} gives the lower rank bound $15+11=26$.

For the matching upper bound, take $\mathscr R$ to consist of the single relation \eqref{eq:Adem22}. Its $t=1$ component on a degree-three domain monomial is $\Sq^2(f)$. For $f=x_i^3$ or $x_i^2x_j$, this component is respectively $x_i^5$ or $x_i^4x_j$, a zero-column label in the $\Sq^2$ layer. The whole relation vector is then in $V_1(3,7)$: after removing its zero-column component, the remaining first-layer vector is killed by $\partial$ and hence is in $\im\Psi_{\Theta^{(1)},7}$. The remaining degree-three domain monomial is $x_1x_2x_3$. Its $t=1$ component is
\[
 x_1^2x_2^2x_3+x_1^2x_2x_3^2+x_1x_2^2x_3^2.
\]
All three labels have nonzero $\Sq^2$ columns, so this component cannot lie in $V_0(3,7)$; $V_1(3,7)$ has no other vectors in the second-layer summand. This gives exactly one additional relation modulo $V_1(3,7)$. Thus $\delta_{\mathscr R}^{(1)}(3,7)=1$ and Theorem~\ref{thm:Adem-upper} gives $\rank M(3,7)\leq26$. Together with the displayed minor this proves
\[
 \rank M(3,7)=26,\qquad \dim(Q\mathcal P_3)_7=10
\]
without elimination of the full hit matrix.

\subsection{The dual bounds in degree six}\label{subsec:degree-six}
The difference between a fixed variable order, all variable orders, and the disjoint-pair family occurs already in degree six. Direct substitution in the formulas gives the following values. The last two columns record the first-square rank and the number of lexicographic leading rows of the projected matrix.
\begin{center}
\begin{tabular}{c|r|r|r|r|r|r}
\toprule
$k$&$N_k(6)$&$z_k(6)$&$\dim \mathcal L_k^{\mathrm{ord}}(6)$&$\beta_k(6)$
&$\rho_k(5)$&$\lambda_{\rm lex}(\overline M(k,6))$\\\midrule
3&28&5&6&6&13&9\\
4&84&18&22&24&35&25\\
5&210&54&64&74&81&55\\
6&462&140&160&190&167&105\\\bottomrule
\end{tabular}
\end{center}
In each row, the last two entries sum to $N_k(6)-\beta_k(6)$. The projected minor bound and the dual bound therefore agree and give $\dim(Q\mathcal P_k)_6=\beta_k(6)$ for these four values of $k$. These dimensions are also instances of the low-degree formulas in \cite[Theorem~2, p.~94, and Table~1, p.~95]{MothebeKMR2016}; the table compares the three specified lower-bound constructions rather than asserting new values for those dimensions.

For $k=4$, the extra two vectors beyond the all-order one-pair space are
\[
 \left(v_1^{(1)}v_2^{(2)}+v_1^{(2)}v_2^{(1)}\right)
 \left(v_3^{(1)}v_4^{(2)}+v_3^{(2)}v_4^{(1)}\right)
\]
and
\[
 \left(v_1^{(1)}v_3^{(2)}+v_1^{(2)}v_3^{(1)}\right)
 \left(v_2^{(1)}v_4^{(2)}+v_2^{(2)}v_4^{(1)}\right).
\]
Their leading sets of exponent-two positions are $\{2,4\}$ and $\{3,4\}$, respectively. They lie in the component with four exponents in $\{1,2\}$ and exactly two exponents equal to two, so their supports are disjoint from those of the one-pair family and the coordinate spikes.

\subsection{Five variables in degree twenty-nine}\label{subsec:k5d29}
The closed formulas give
\[
 \begin{gathered}
 N_5(29)=40920,\quad \mathcal S_5(29)=95,\quad
 \mathcal W_5(29)=99961,\quad \mathcal W_5^{(1)}(29)=86220,\\
 D_5^\cup(29)=1050,\quad
 B_{0,2}=F_{0,2}=11424,\quad
 B_{0,4}=7735,\quad F_{0,4}=5005.
 \end{gathered}
\]
The simplicial sum is $B_0(5,29)=19159=\rho_5(28)$. Theorem~\ref{thm:all-singletons} gives $c_1^{\mathrm{sing}}(5,29)=5250$ and $\sigma_1(5,29)=1650$. All other layers have no singleton columns. Thus the complete singleton target count is $1650$, while the selected congruence family has $1050$ targets. The complete original-layer enumeration is
\begin{center}
\begin{tabular}{r|r|r|r|r}
\toprule
$t$&Domain monomials&Nonzero columns&New lex leading rows&Cumulative\\\midrule
0&35960&32900&19159&19159\\
1&31465&31465&7656&26815\\
2&23751&23576&2048&28863\\
3&12650&12020&414&29277\\
4&2380&0&0&29277\\\bottomrule
\end{tabular}
\end{center}
For all layers together, the four-order counts are
\begin{center}
\begin{tabular}{c|r|r}
\toprule
Order&$\lambda_\prec(M(5,29))$&$\lambda_\prec(\overline M(5,29))$\\\midrule
lexicographic&29277&16136\\
degree-reverse-lexicographic&31524&17028\\
left weight&39086&20096\\
right weight&30437&15212\\\bottomrule
\end{tabular}
\end{center}
All projected counts concern only $t\geq1$. There are $21666$ nonzero rows of $\overline M$ in the $21761$-dimensional space $T_{29}(5)$. Its $95$ zero rows give only the coordinate spike functionals in this example, as also follows from \eqref{eq:closed-projected-zero-count}: $\mathcal S_s(26)=0$ for $0\leq s\leq3$. The improved lower hit rank is $19159+20096=39255$, and the spike upper rank is $40920-95=40825$. Therefore the selected minors and dual functionals prove
\begin{equation}\label{eq:k5d29-interval}
 39255\leq\rank M(5,29)\leq40825,
 \qquad 95\leq\dim(Q\mathcal P_5)_{29}\leq1665.
\end{equation}
The best unprojected order in the table is left weight, giving $40920-39086=1834$. The projection improves this upper bound by $1834-1665=169$. The larger change from the exponent-lexicographic bound $11643$ also includes the effect of changing the row order.

In this degree $\beta_5(29)=95$ as well, so the disjoint-pair extension does not increase the lower endpoint obtained from the projected zero rows. Mothebe--Uys \cite[Theorem~6 and Table~1]{MothebeUys2015} give the stronger lower bound $450$ by inserting spike exponents into admissible monomials. The same insertion principle, evaluated with the following specified four-variable bases, gives a further improvement. We use the bases $\mathcal B_4(n)$ and the insertion sets $\mathcal I_\theta$ fixed in Section~\ref{subsec:prelim-combined}. 
The required four-variable dimensions follow from \cite[Theorem~4, p.~567]{Sum2013}, with the Kameko reduction \cite[\S6.5, Proposition~6.5.3, p.~111]{WalkerWoodI} for degrees $28$ and $26$. In particular, degrees $28$ and $26$ reduce to degrees $12$ and $11$, respectively, and degree $12$ reduces further to degree $4$. 
The finite procedure in Section~\ref{subsec:reproducibility} specifies these bases and the union counts.

By \cite[Theorem~4]{MothebeUys2015}, every monomial of $\mathcal I_\theta$ is admissible in five variables. Indeed, the hypothesis $\alpha(29-\theta+4)\leq4$ holds for all five values of $\theta$: the corresponding integers are $33,32,30,26,18$, with binary-one counts $2,1,4,3,2$. Distinct admissible monomials are linearly independent for this fixed order. The required calculation is therefore a count of the union, not a sum of the five separate cardinalities. In increasing order of $\theta$, the counts are
\begin{center}
\begin{tabular}{r|r|r|r|r}
\toprule
$\theta$&$29-\theta$&$|\mathcal B_4(29-\theta)|$&$|\mathcal I_\theta|$
  &$|\mathcal I_\theta\setminus\bigcup_{\theta'<\theta}\mathcal I_{\theta'}|$\\\midrule
0&29&45&195&195\\
1&28&21&105&105\\
3&26&64&175&75\\
7&22&116&445&130\\
15&14&50&250&45\\\bottomrule
\end{tabular}
\end{center}
Each entry is obtained by inserting $\theta$ into the four-component tuples and deleting repeated five-component tuples. Thus
\begin{equation}\label{eq:k5d29-strengthened}
 550\leq\dim(Q\mathcal P_5)_{29}\leq1665.
\end{equation}
This lower bound uses the classical admissible-monomial insertion theorem in addition to the support calculations; it is not a count of zero rows of $\overline M(5,29)$.

For comparison, the dimension formula in \cite[Theorem~1.1]{Phuc2020}, specialized to degree $3(2^3-1)+2^3=29$, together with \eqref{eq:dimension-rank}, gives
\begin{equation}\label{eq:exact29}
 \rank M(5,29)=40275,\qquad \dim(Q\mathcal P_5)_{29}=645.
\end{equation}

\subsection{Reproducible finite calculations}\label{subsec:reproducibility}
The numerical tables use the following finite procedures over $\F$. For each layer $t$, enumerate the weak compositions $a$ of $d-2^t$. For each $a$, enumerate the coordinatewise binary subsets $b_i\preceq a_i$ with $\sum_i b_i=2^t$. Their distinct outputs $a+b$ are exactly the support of the corresponding column. Taking the largest output in each of the four orders and forming the union of these outputs gives the four leading-row counts. For the cumulative counts, process increasing $t$ and retain the previous leading rows. A column with one output contributes it to the singleton-target set.

For a projected column, enumerate instead the allocations with $\sum_i b_i=2^t+1$. If $a=0$, the column is zero and is discarded. Otherwise let $i=i(a)$ be the first positive position of $a$, and retain exactly the outputs $a+b$ whose $i$th entry is even and replace that output by $a+b-\eps_i$. As proved after Theorem~\ref{thm:Sq1-quotient}, the retained outputs are distinct. The same four leading-row counts follow. These calculations enumerate the specified supports without reducing the full hit matrix.

The insertion calculation uses the full four-variable relation space in the five stated degrees. Order its monomials increasingly by left weight and then by exponent lexicographic order. Process the generating-square columns in increasing $t$ and, within each layer, increasing exponent lexicographic order of the domain monomials. For a nonzero column, let its largest row be its pivot. If a retained column has the same pivot, add that retained column and repeat. If no retained column has that pivot, retain the new column; discard a zero result. Each addition strictly decreases the largest row, so the process terminates.

The retained columns have distinct pivots and span the hit space by induction on the processed columns. A monomial at a pivot is congruent modulo hits to a combination of smaller monomials. A nonpivot monomial cannot have this property: such a relation would be a hit vector whose largest row is a nonpivot, whereas the largest row of any nonzero combination of retained columns is one of their pivots. Consequently the nonpivot monomials are exactly the admissible basis for the fixed order. In degrees $29,28,26,22,14$, the resulting four-variable ranks are respectively $4915,4474,3590,2184,630$; subtracting from the corresponding monomial counts gives $45,21,64,116,50$.

Insert each prescribed spike exponent into each of the five positions of these nonpivot tuples and take their union. The successive new tuples number $195,105,75,130,45$, whose sum is $550$. This specifies the computation underlying \eqref{eq:k5d29-strengthened}; the linear independence of its $550$ distinct outputs in the cohit quotient follows from the admissible-monomial insertion theorem of Mothebe and Uys \cite[Theorem~4]{MothebeUys2015}.

\medskip

\section*{Funding}
The first author was funded by the Post-Doctoral Scholarship Programme of Vingroup Innovation Foundation (VINIF), Institute of Big Data, code: VINIF.2024.STS.38.

%\section*{Conflict of interest}
%The authors declare no conflict of interest.

\section*{Declaration on AI usage and formal verification in Lean 4}

All ideas, the formulation of all results, the complete proofs, and the final text were developed, verified, and revised by the human authors, who assume full responsibility for the entire manuscript. AI-assisted tools were used during the preparation of this manuscript to refine the exposition and notation and to assist in drafting the accompanying scripts for formal verification in Lean~4.

\medskip

The two-sided rank bounds, contraction and projection identities, closed combinatorial formulas, and dual-family linear independence results in Sections~\ref{sec:upper}--\ref{sec:main}, including the rank and quotient-dimension bounds corresponding to Theorem~\ref{thm:main-rank}, have been formalized in \textit{Lean~4.36.0} with \textit{Mathlib~4.36.0} for the explicitly defined Cartan--Lucas generating-square matrices, their column spans and quotients, and the associated dual functionals. The archived build completed without errors or warnings. The source contains no ``\texttt{sorry}'' or ``\texttt{admit}'' placeholders and no additional axioms beyond Lean's standard kernel. The verification report records the correspondence between the manuscript statements and the formal declarations, together with their precise hypotheses and scope.

\medskip

The general arguments are expressed as Lean proofs. Closed finite checks additionally use ``\texttt{decide}'', ``\texttt{norm\_num}'', and ``\texttt{native\_decide}''. In Lean~4.36.0, native evaluation introduces automatically generated axioms recording the results of the corresponding computations. The finite results that use this mechanism therefore depend on the correctness of the native compiler and evaluator in addition to Lean's kernel and standard foundational axioms. The accompanying axiom-dependency audit records these generated dependencies separately from the standard axioms and confirms the absence of admitted proofs and unexpected axioms.

\medskip

The formal cohit space is defined as the quotient by the column span of the labelled Cartan--Lucas generating-square matrix. Its identification with the corresponding graded component of the classical quotient $\Pk/\Aplus\Pk$ uses the theorem that the squares $\Sq^{2^t}$ generate $\A$ \cite[\S3.2, Proposition~3.2.4]{WalkerWoodI}; this general identification remains outside the formal development. Other cited background results are also kept distinct from the formalized constructions. These include the general admissible-monomial insertion theorem of Mothebe and Uys \cite[Theorem~4]{MothebeUys2015}, the four-variable dimension formulas of Sum \cite[Theorem~4]{Sum2013}, the Kameko reduction \cite[\S6.5, Proposition~6.5.3]{WalkerWoodI}, the low-degree formulas of Mothebe, Kaelo, and Ramatebele \cite[Theorem~2]{MothebeKMR2016}, and the five-variable generic-degree formula in \cite[Theorem~1.1]{Phuc2020}.

\medskip

The displayed finite examples and numerical tables are checked directly in Lean, including the four-variable rank calculations and the degree-$29$ calculations underlying \eqref{eq:k5d29-strengthened} and \eqref{eq:exact29}. In particular, a proved binary elimination procedure and checked finite certificates establish the linear independence of the specific $550$ inserted classes in the quotient by the degree-$29$ Cartan--Lucas column space. This finite independence proof does not assume the general Mothebe--Uys insertion theorem \cite[Theorem~4]{MothebeUys2015}. The same formal development verifies the hit-matrix rank $40275$ and the corresponding matrix-quotient dimension $645$. Independent exact-arithmetic computations over $\F$ implemented in Python provide additional corroboration of the reported numerical data.

\medskip

The accompanying Lean sources and independent Python verification scripts are available in the archive ``\texttt{PetersonHitBounds-Lean4-Zenodo.zip}'' on Zenodo, DOI:  \href{https://doi.org/10.5281/zenodo.24548499}{10.5281/zenodo.24548499}. The archive contains the directories ``\texttt{Code Lean}'' and ``\texttt{Code python}'', the pinned toolchain and dependency configuration, instructions for reproducing the build and finite calculations, the statement-to-declaration traceability record, and the build and axiom-dependency audits. The included ``\texttt{SHA256SUMS}'' file records the SHA-256 checksums of the distributed files, excluding the checksum file itself.

\end{document}